\documentclass{article}
\usepackage{amsmath,amsfonts,amsthm,amssymb, graphicx}
\usepackage{subfig}
\usepackage{amsmath}
\usepackage{amsfonts}
\usepackage{amssymb}
\usepackage{graphicx}
\usepackage{color}
\usepackage{mathrsfs}
\usepackage{geometry}
\usepackage{enumitem}
\usepackage{float}

\newtheorem{theorem}{Theorem}

\newtheorem{lemma}[theorem]{Lemma}

\DeclareMathOperator*{\argmin}{arg\,min}

\newtheorem{asu}{Assumption}
\newcounter{subassumption}[asu]
\renewcommand{\thesubassumption}{(\textit{\roman{subassumption}})}
\makeatletter
\renewcommand{\p@subassumption}{\theasu}
\makeatother

\begin{document}

\title{Unbiased Multilevel Monte Carlo: Stochastic Optimization, Steady-state Simulation, Quantiles, and Other Applications}
\author{Jose H. Blanchet\thanks{%
Department of Management Science and Engineering, Stanford University,
Stanford, CA 94305} \and Peter W. Glynn\footnotemark[1] \and Yanan Pei%
\thanks{%
Department of Industrial Engineering and Operations Research, Columbia
University, NY 10027}}
\date{}
\maketitle

\begin{abstract}
We present general principles for the design and analysis of unbiased Monte
Carlo estimators in a wide range of settings. Our estimators posses finite
work-normalized variance under mild regularity conditions. We apply our
estimators to various settings of interest, including unbiased optimization
in \textit{Sample Average Approximations}, unbiased steady-state simulation
of regenerative processes, quantile estimation and nested simulation
problems.
\end{abstract}


\section{Introduction}

In this paper, we propose simple and yet powerful techniques that can be
used to delete bias that often arise in the implementation of Monte Carlo
computations in a wide range of decision making and performance analysis
settings, for instance, stochastic optimization and quantile estimation,
among others.

There are two key advantages of the estimators that we will present.
Firstly, they can be easily implemented in the presence of parallel
computing processors, yielding estimates whose accuracy improves as the size
of available parallel computing cores increase while keeping the
work-per-processor bounded in expectation. Secondly, the confidence
intervals can be easily produced in settings in which variance estimators
might be difficult to obtain (e.g. in stochastic optimization problems whose
asymptotic variances depend on Hessian information).

To appreciate the advantage of parallel computing with bounded cost per
parallel processor, let us consider a typical problem in machine learning
applications, which usually involves a sheer amount of data. Because of the
technical issues that arise in using the whole data set for training, one
needs to resort techniques such as stochastic gradient descent, which is not
easy to fully run in parallel, or sample-average approximations, which can
be parallelized easily but it carries a systematic bias. For both the
optimal value function and the optimal policies, we provide estimators that
are unbiased, possess finite variance, and can be implemented in finite
expected termination time. Thus, our estimators can be directly implemented
in parallel with each parallel processor being assigned an amount of work
which is bounded in expectation.

A second example that we shall consider as an application of our techniques
arises in steady-state analysis of stochastic systems. A typical setting of
interest is to compute a long-term average of expected cost or reward for
running a stochastic system. This problem is classical in the literature of
stochastic simulation and it has been studied from multiple angles. Our
simple approach provides another way such that the steady-state analysis of
regenerative processes can be done without any bias. A key characteristic is
that the approach we study involves the same principle underlying the
stochastic optimization setting mentioned in the previous paragraph.

Other types of problems that we are able to directly address using our
methodology include computing unbiased estimators of quantiles and unbiased
estimators of nested expectations. In addition to being unbiased, all of the
estimators have finite work-normalized variance and can be simulated in
finite expected termination time, which makes their implementation in
parallel computation straightforward.

Applications such as stochastic optimization and quantile estimation allow
us to highlight the fact that the variance estimates of our Monte Carlo
estimators are straightforward to produce. These variance estimates are
important for us to generate asymptotically accurate confidence intervals.
In contrast, even though asymptotically unbiased estimators may be
available, sometimes these estimators require information about Hessians (as
in the optimization setting) or even density information (as in quantile
estimation applications) to produce accurate confidence intervals, while our
estimators do not require this type of information.

Our estimator relates to the multilevel-Monte Carlo method developed in \cite%
{Giles_2008}. We apply the debiasing techniques introduced in \cite%
{GlynnRhee_2015} and \cite{McLeish_2012}. Since the introduction of these
techniques, several improvements and applications have been studied, mostly
in the context of stochastic differential equations and partial differential
equations with random input, see for example \cite{GilesSzpruch_2014}, \cite%
{AgarwalGobet_2017}, \cite{KTH_2018} and \cite{MIMC_2018}.

In \cite{Vihola_2018}, a stratified sampling technique is introduced in
order to show that the debiasing in multi-level Monte Carlo can be achieved
virtually at no cost in either asymptotic efficiency or sample complexity
relative to the standard (biased) MLMC estimator. The results of \cite%
{Vihola_2018} can be applied directly to our estimators in order to improve
the variance, but the qualitative rate of convergence (i.e. O$\left(
1/\varepsilon ^{2}\right) $ remains the same).

A recent and independent paper \cite{GilesHA_2018} also studies the nested
simulation problem. While the techniques that we propose here are very
related to theirs, our framework is postulated in greater generality, which
involves not only nested simulation problems but other applications as
described earlier as well.

Another recent paper \cite{Dereich_2017} studies multi-level Monte Carlo in
the context of stochastic optimization, but their setting is different from
what we consider here and they give a completely different class of
algorithms which are not unbiased. Another recent paper \cite{Bujok_2015}
also develops optimal O($1/\varepsilon ^{2}$) sample complexity algorithms
for nested simulation problems, but they focus on specific settings and do
not develop a framework as general as what we consider here.

The rest of the paper is organized as follows. In Section \ref%
{sec:general-principle} we discuss the general principle which drives the
construction of our unbiased estimators. Then, we apply these principles to
the different settings of interest, namely, unbiased estimators for
non-linear functions of expectations, stochastic convex optimization,
quantile estimation and nested simulation problems, in later sections.

\section{The General Principles}

\label{sec:general-principle}

The general principles are based on the work of \cite{GlynnRhee_2015}.
Suppose that one is interested in estimating a quantity of the form $\theta
\left( \mu \right) \in \mathbb{R}$, where $\mu $ is a generic probability
distribution, say with support in a subset of $\mathbb{R}^{d}$, and $\theta
\left( \cdot \right) $ is a non-linear map.

A useful example to ground the discussion in the mind of the reader is $%
\theta \left( \mu \right) =g\left( E_{\mu }\left[ X\right] \right) $, where $%
g:\mathbb{R}^{d}\rightarrow \mathbb{R}$ is a given function (with regularity
properties which will be discussed in the sequel). We use the notation $%
E_{\mu }\left( \cdot \right) $ to denote the expectation operator under the
probability distribution $\mu $. For the sake of simplicity, we will later
omit the subindex $\mu $ when the context is clear.

We consider the empirical measure $\mu_n$ of iid samples $\left\{ X_{i}\in%
\mathbb{R}^d:1\le i\le n\right\} $, i.e.,
\begin{equation*}
\mu _{n}\left( dx\right) =\frac{1}{n}\sum_{i=1}^{n}\delta _{\left\{
X_{i}\right\} }\left( dx\right),
\end{equation*}
where $\delta_{X_i}(\cdot)$ is the point mass at $X_i$ for $i=1,\ldots,n$.
The sample complexity of producing $\mu _{n}\left( \cdot \right) $ is equal
to $n$. By the strong law of large numbers for empirical measures
(Varadarajan's Theorem), $\mu _{n}\rightarrow \mu $ almost surely in the
Prohorov space.

Under mild continuity assumptions, we have that
\begin{equation}
\theta \left( \mu _{n}\right) \rightarrow \theta \left( \mu \right)
\label{Aunb}
\end{equation}%
as $n\rightarrow \infty $ and often one might expect that $\theta \left( \mu
_{n}\right) $ is easy to compute. Then, $\theta \left( \mu _{n}\right) $
becomes a natural and reasonable estimator for $\theta \left( \mu \right) $.
However, there are several reasons that make it desirable to construct an
unbiased estimator with finite variance, say $Z$, for $\theta \left( \mu
\right) $; even if $\left\{ \theta \left( \mu _{n}\right) \right\} _{n\geq
1} $ is asymptotically normal in the sense that $n^{1/2}\left( \theta \left(
\mu _{n}\right) -\theta \left( \mu \right) \right) \Rightarrow \mathcal{N}%
\left( 0,\sigma _{\theta }^{2}\right) $ with some $\sigma _{\theta }^{2}>0$.
First, as we mentioned in the introduction, if one copy of $Z$ can be
produced in finite expected time, averaging the parallel replications of $Z$
immediately yields an estimate of $\theta \left( \mu \right) $, whose
accuracy then can be increased by the Central Limit Theorem as the number of
parallel replications of $Z$ increases. Second, the variance of $Z$ can be
estimated with the natural variance estimator of iid replications of $Z$,
when $\sigma _{\theta }^{2}$ may be difficult to evaluate from the samples
(e.g. if $\theta \left( \mu \right) $ represents some quantile of $\mu$).

Our goal is to construct a random variable $Z$ such that
\begin{equation*}
E\left[ Z\right] =\theta \left( \mu \right),\ Var\left( Z\right) <\infty,
\end{equation*}
and the expected sample complexity to produce $Z$ is bounded. To serve this
goal, we first construct a sequence of random variables $\left\{ \Delta
_{m}:m\ge1\right\}$ satisfying the following properties:

\begin{asu}
\label{general-principle-assumptions} General assumptions\newline

\refstepcounter{subassumption} \thesubassumption~\ignorespaces\label%
{asu-general1} There exists some $\alpha,c\in \left( 0,\infty \right)$ such
that $E\left[\left\vert \Delta _{m}\right\vert ^{2}\right] \leq c\cdot
2^{-(1+\alpha)m}$,\newline

\refstepcounter{subassumption} \thesubassumption~\ignorespaces\label%
{asu-general2} $\sum_{m=0}^{\infty }E\left[ \Delta _{m}\right] =\theta
\left( \mu\right) $,\newline

\refstepcounter{subassumption} \thesubassumption~\ignorespaces\label%
{asu-general3} If $C_{m}$ is the computational cost of producing one copy of
$\Delta _{m}$ (measured in terms of sampling complexity), then $E\left[ C_{m}%
\right] \leq c^{\prime }\cdot 2^{m}$ for some $c^{\prime }\in \left(
0,\infty \right) $.
\end{asu}

If we are able to construct the sequece $\left\{ \Delta _{m} :m\ge1 \right\}$
satisfying Assumption \ref{general-principle-assumptions}, then we can
construct an unbiased estimator for $Z$ as follows. First, sample $N$ from
geometric distribution with success parameter $r$, so that $p\left( k\right)
=P\left( N=k\right) =r\left( 1-r\right) ^{k}$ for $k\ge 0$. The parameter $%
r\in \left( 0,1\right) $ will be optimized shortly. At this point it
suffices to assume that $r\in \left( \frac{1}{2},1-\frac{1}{2^{(1+\alpha)}}%
\right) $.

Once the distribution of $N$ has been specified, the estimator that we
consider takes the form
\begin{equation}  \label{general-estimator}
Z=\frac{\Delta _{N}}{p\left( N\right) },
\end{equation}
where $N$ is independent of the iid sequence $\left\{ \Delta _{m}\right\}
_{m=1}^{\infty }$. Note that the estimator possess finite variance because $%
r<1-\frac{1}{2^{(1+\alpha)}}$,
\begin{eqnarray}
E\left[ Z^{2}\right] &=&\sum_{k=0}^{\infty }E\left[ Z^{2}\mid N=k\right]
p\left( k\right) =\sum_{k=0}^{\infty }E\left[ \frac{\Delta _{k}^{2}}{%
p\left(k\right)^2}\mid N=k\right]p\left( k\right)  \label{eA} \\
&=&\sum_{k=0}^{\infty}\frac{E\left[ \Delta _{k}^{2}\right]}{p\left( k\right)
} \leq c\sum_{k=0}^{\infty }\frac{2^{-(1+\alpha)k}}{p\left( k\right) }=\frac{%
c}{r}\sum_{k=0}^{\infty }\frac{1}{\left( 2^{(1+\alpha)}\left( 1-r\right)
\right) ^{k}}<\infty .  \notag
\end{eqnarray}%
Moreover, the unbiasedness of the estimator is ensured by Assumption \ref%
{asu-general2},
\begin{equation*}
E\left[ Z\right] =\sum_{k=0}^{\infty }E\left[ Z\mid N=k\right] p\left(
k\right) =\sum_{k=0}^{\infty }E\left[ \frac{\Delta _{k}}{p\left( k\right) }%
\right] p\left( k\right) =\theta \left( \mu \right) .
\end{equation*}%
Finally, because $r>1/2$, the expected sampling complexity of producing $Z$,
denoted by $C$, is finite precisely by Assumption \ref{asu-general3},%
\begin{equation}
E\left[C\right]=E\left[ C_{N}\right] \leq c^{\prime }\sum_{k=0}^{\infty
}2^{k}p\left( k\right) =rc^{\prime }\sum_{k=0}^{\infty }\left( 2\left(
1-r\right) \right) ^{k}<\infty .  \label{eC}
\end{equation}%
In \cite{GlynnRhee_2015} the choice of $N$ is optimized in terms of the $E%
\left[ \Delta _{m}^{2}\right] $ and $E\left[ C_{m}\right] $. In \cite%
{BlanGlynn:2015} a bound on the work-normalized variance, namely the product%
\begin{equation*}
\sum_{k=0}^{\infty }\frac{2^{-(1+\alpha)k}}{p\left( k\right) }\times
\sum_{k=0}^{\infty }2^{k}p\left( k\right) ,
\end{equation*}%
corresponding to the bounds in the right hand side of (\ref{eA}) and (\ref%
{eC}) is optimized, and the resulting optimal choice of $p\left( k\right) $%
's corresponds to choosing $N$ geometrically distributed with $r=1-2^{-3/2}$
when $\alpha=1$. Following the same logic, the optimal choice of $N$ should
be geometrically distributed with $r=1-2^{-(1+\alpha/2)}$ for the general $%
\alpha>0$ case and we advocate this choice for the construction of $Z$.

The contribution of our work is to study the construction of the $\Delta
_{m} $'s based on the sequence $\left\{ \mu _{n}:n\geq 1\right\} $
satisfying Assumption \ref{general-principle-assumptions} as we\ now
explain. Now our focus is on explaining the high-level ideas at an informal
level and provide formal assumptions later for different settings.

Suppose that there exists a function $T_{\mu }^{\theta }:\mathbb{R}%
^{d}\rightarrow \mathbb{R}$ such that
\begin{equation*}
\left. \frac{d}{dt}\theta \left( \mu +t\left( \mu _{n}-\mu \right) \right)
\right\vert _{t=0}=\int T_{\mu }^{\theta }\left( x\right) d\left( \mu
_{n}-\mu \right) =E_{\mu _{n}}\left[ T_{\mu }^{\theta }\left( X\right) %
\right] -E_{\mu }\left[ T_{\mu }^{\theta }\left( X\right) \right].
\end{equation*}%
Typically, $T_{\mu }^{\theta }\left( \cdot \right) $ corresponds to the
Riesz representation (if it exists) of the derivative of $\theta \left(
\cdot \right) $ at $\mu $. Going back to the case in which $\theta \left(
\mu \right) =g\left( E_{\mu }\left[ X\right] \right) $, assuming that $%
g\left( \cdot \right) $ is differentiable with derivative $Dg\left( \cdot
\right) $, we have
\begin{equation*}
\left. \frac{d}{dt}g\left( E_{\mu }\left[ X\right] +t\left( E_{\mu _{n}}%
\left[ X\right] -E_{\mu }\left[ X\right] \right) \right) \right\vert
_{t=0}=Dg\left( E_{\mu }\left[ X\right] \right) \cdot \left( E_{\mu _{n}}%
\left[ X\right] -E_{\mu }\left[ X\right] \right) ,
\end{equation*}%
so in this setting $T_{\mu }^{\theta }\left( x\right) =Dg\left( \int z\mu
\left( dz\right) \right) \cdot x=Dg\left( E_{\mu }\left[ X\right] \right)
\cdot x$.

Now, suppose that $\theta \left( \cdot \right) $ is smooth in the sense that
\begin{equation}
\theta \left( \mu \right) =\theta \left( \mu _{n}\right) +\left( E_{\mu }
\left[ T_{\mu }^{\theta }\left( X\right) \right] -E_{\mu _{n}}\left[ T_{\mu
}^{\theta }\left( X\right) \right] \right) +\epsilon \left( n,\mu
_{n}\right) ,  \label{Smooth}
\end{equation}%
where
\begin{equation*}
\left\vert \epsilon \left( n,\mu _{n}\right) \right\vert =O_{p}\left( \left|
E_{\mu }\left[ T_{\mu }^{\theta }\left( X\right) \right] -E_{\mu _{n}}\left[
T_{\mu }^{\theta }\left( X\right) \right] \right| ^{2}\right) .
\end{equation*}%
Basically, the term $\epsilon \left( n,\mu _{n}\right) $ controls the error
of the first order Taylor expansion of the map
\begin{equation*}
t\hookrightarrow \theta \left( \mu +t\left( \mu _{n}-\mu \right) \right)
\end{equation*}
around $t=0$. So, in the context in which $\theta \left( \mu \right)
=g\left( E_{\mu }\left[ X\right] \right) $, if $g\left( \cdot \right) $ is
twice continuously differentiable, then we have
\begin{eqnarray*}
&&\epsilon \left( n,\mu _{n}\right) \\
&=&\frac{1}{2}\left( E_{\mu _{n}}\left[ X\right] -E_{\mu }\left[ X\right) %
\right] ^{T}\cdot \left( D^{2}g\right) \left( E_{\mu }\left[ X\right]
\right) \cdot \left( E_{\mu _{n}}\left[ X\right] -E_{\mu }\left[ X\right]
\right) +o\left( 1\right) ,
\end{eqnarray*}%
as $n\rightarrow \infty $.

The key ingredient in the construction of the sequence $\left\{ \Delta
_{m}:m\geq 1\right\}$ is an assumption of the form 
\begin{equation}
\sup_{n\geq 1}n^{2} E\left[\left| E_{\mu }\left[ T_{\mu }^{\theta }\left(
X\right) \right] -E_{\mu _{n}}\left[ T_{\mu }^{\theta }\left( X\right) %
\right] \right| ^{4}\right] <\infty ,  \label{L2Bnd}
\end{equation}%
this assumption will typically be followed as an strengthening of a Central
Limit Theorem companion to the limit $\mu _{n}\rightarrow \mu $ as $%
n\rightarrow \infty $, which would typically yield 
\begin{equation*}
n^{1/2}\left\{ E_{\mu }\left[ T_{\mu }\left( X\right) \right] -E_{\mu _{n}}%
\left[ T_{\mu }\left( X\right) \right] \right\} \Longrightarrow W,
\end{equation*}%
as $n\rightarrow \infty $ for some $W$. Under (\ref{L2Bnd}) the construction
of $\Delta _{n}$ satisfying Assumption \ref{asu-general1} proceeds as
follows. Let
\begin{equation*}
\mu _{2^{n}}^{E}\left( dx\right) =\frac{1}{2^n}\sum_{i=1}^{2^{n}}\delta
_{\{X_{2i}\}}\left( dx\right),~~~~\mu _{2^{n}}^{O}\left( dx\right) =\frac{1}{%
2^n}\sum_{i=1}^{2^{n}}\delta _{\{X_{2i-1}\}}\left( dx\right)
\end{equation*}
and set for $n\geq 1$,
\begin{equation}
\Delta _{n}=\theta \left( \mu _{2^{n+1}}\right) -\frac{1}{2}\left(\theta
\left( \mu _{2^{n}}^{E}\right) +\theta \left( \mu _{2^{n}}^{O}\right)
\right).  \label{Estimator}
\end{equation}%
The key property behind the construction for $\Delta _{n}$ in (\ref%
{Estimator}) is that
\begin{equation*}
\mu _{2^{n+1}}=\frac{1}{2}\left( \mu _{2^{n}}^{E}+\mu _{2^{n}}^{O}\right) ,
\end{equation*}%
so a linearization of $\theta \left( \mu \right) $ will cancel the first
order effects implied in approximating $\mu $ by $\mu_{2^{n+1}}$,$\mu
^O_{2^{n}}$ and $\mu^E_{2^n}$. In particular, using (\ref{Smooth}) directly
we have that
\begin{equation*}
\left\vert \Delta _{n}\right\vert \leq \left\vert \epsilon \left(
2^{n+1},\mu _{2^{n+1}}\right) \right\vert +\left\vert \epsilon \left(
2^{n},\mu _{2^{n}}^{O}\right) \right\vert +\left\vert \epsilon \left(
2^{n},\mu _{2^{n}}^{E}\right) \right\vert ,
\end{equation*}%
consequently due to (\ref{L2Bnd}) we have that
\begin{equation}
E\left[ \left\vert \Delta _{n}\right\vert ^{2}\right] =O\left(
2^{-2n}\right) .  \label{KeyB2}
\end{equation}%
Once (\ref{KeyB2}) is in place, verification of Assumption \ref{asu-general2}
is straightforward because
\begin{equation*}
E\left[ \Delta _{n}\right] =E\left[ \theta \left( \mu _{2^{n+1}}\right) %
\right] -E\left[ \theta \left( \mu _{2^{n}}\right) \right] ,
\end{equation*}%
so if we define
\begin{equation*}
\Delta _{0}=\theta \left( \mu _{2}\right) ,
\end{equation*}%
then
\begin{equation*}
\sum_{n=0}^{\infty }E\left( \Delta _{n}\right) =\theta \left( \mu \right) .
\end{equation*}%
Assumption \ref{asu-general3} follows directly because the sampling
complexity required to produce $\Delta _{m}$ is $C_{m}=2^{m+1}$ (assuming
each $X_{i}$ required a unit of sample complexity).

The rest of the paper is dedicated to the analysis of (\ref{Estimator}). The
abstract approach described here, in terms of the derivative of $\theta
\left( \mu \right)$, sometimes is cumbersome to implement under the
assumptions that are natural in the applications of interest (for example
stochastic optimization). So, we may study the error in (\ref{Estimator})
directly in later applications, but we believe that keeping the high-level
intuition described here is useful to convey the generality of the main
ideas.

\section{Non-linear functions of expectations and applications}

\label{sec:function}

\label{sec:functions-expecations}

We first apply the general principle to the canonical example considered in
our previous discussion, namely
\begin{equation*}
\theta (\mu )=g\left( \int yd\mu (y)\right) =g\left( E_{\mu }\left[ X\right]
\right) .
\end{equation*}%
Let $\nu=E_{\mu}[X]$. We will impose natural conditions on $g\left( \cdot
\right) $ to make sure that the principles discussed in Section \ref%
{sec:general-principle} can be directly applied.

We use $(X_k:k\ge1)$ to denote an iid sequence of copies of the random
variable $X\in\mathbb{R}^d$ from distribution $\mu$. For $k\ge1$, we define
\begin{equation*}
X_k^O=X_{2k-1}~~~~\mbox{and}~~~~X_k^E=X_{2k}.
\end{equation*}
Note that the $X^O$'s correspond to $X_k$'s indexed by odd values and the $%
X^E$'s correspond to the $X_k$'s indexed by even values. For $k\in\mathbb{N}%
_+$, let
\begin{equation*}
S_k=X_1+\ldots+X_k
\end{equation*}
and similarly let
\begin{align*}
S_k^O&=X_1^O+\ldots+X_k^O, \\
S_k^E&=X_1^E+\ldots+X_k^E.
\end{align*}

In this setting, we may define%
\begin{equation*}
\Delta _{n}=g\left( \frac{S_{2^{n+1}}}{2^{n+1}}\right) -\frac{1}{2}\left(
g\left( \frac{S_{2^{n}}^{O}}{2^{n}}\right) +g\left( \frac{S_{2^{n}}^{E}}{%
2^{n}}\right) \right)
\end{equation*}%
for $n\geq 0$ and let the estimator to be
\begin{equation}
Z=\frac{\Delta _{N}}{p\left( N\right) }+g\left( X_1\right) ,
\end{equation}%
where $N$ was defined in Section \ref{sec:general-principle}.

We now impose precise assumptions on $g\left( \cdot \right) $, so that
Assumption \ref{general-principle-assumptions} can be verified for $\Delta
_{n}$. We summarize our discussion in Theorem \ref{thm-function-expectation}
next.

\begin{theorem}
\label{thm-function-expectation}Suppose that the following assumptions are
forced:

\begin{enumerate}
\item Suppose that $g:\mathbb{R}^d\rightarrow\mathbb{R}$ has linear growth
of the form $\lvert g(x)\rvert\le c_1\left(1+\lVert x\rVert_2\right)$ for
some $c_1>0$, where $\lVert\cdot\rVert_2$ denotes the $l_2$ norm in
Euclidian space, \label{function-thm-1}

\item Suppose $g$ is continuously differentiable in a neighborhood of $%
\nu=E[X]$, and $Dg(\cdot)$ is locally Holder continuous with exponent $%
\alpha>0$, i.e.,

\begin{equation*}
\left\Vert Dg(x)-Dg(y)\right\Vert _{2}\leq \kappa (x)\left\Vert
x-y\right\Vert _{2}^{\alpha },
\end{equation*}%
where $\kappa \left( \cdot \right) $ is bounded on compact sets \label%
{function-thm-2}


\item $X$ has finite $3(1+\alpha )$ moments, i.e. $E\left[ \lVert X\rVert
_{2}^{3(1+\alpha )}\right] <\infty $. \label{function-thm-3}
\end{enumerate}

Then, $E\left[Z\right]=g\left( E\left[ X\right] \right) $, $Var\left(
Z\right) <\infty $ and the sampling complexity required to produce $Z$ is
bounded in expectation.
\end{theorem}

\begin{proof}
We first show the unbiasedness of the estimator $Z$. From \ref%
{function-thm-1} we have that
\begin{equation*}
\left| g\left(S_n/n\right)\right|^2\le c_1^{\prime }\left(1+\lVert
S_n/n\rVert_2^{2}\right).
\end{equation*}
And because of \ref{function-thm-3} we have $E\left[g(S_n/n)^2\right]%
<\infty, $ which implies that $g(S_n/n)$ is uniformly integrable. For each $%
n\ge0$,
\begin{equation*}
E\left[\Delta_n\right]=E\left[g\left(S_{2^{n+1}}/2^{n+1}\right)\right]-E%
\left[g\left(S_{2^n}/2^n\right)\right].
\end{equation*}
With the condition in \ref{function-thm-2} that $g$ is continuous in a
neighborhood of $\nu$, we derive
\begin{align*}
E[Z]&=E\left[\frac{\Delta_N}{p(N)}\right]+E\left[g(X_1)\right]%
=\sum_{n=1}^{\infty}E\left[\Delta_n\right]+E\left[g(X_1)\right] \\
&=\lim_{n\to\infty}E\left[g\left(S_{2^n}/2^n\right)\right]=E\left[%
\lim_{n\to\infty}g\left(S_{2^n}/2^n\right)\right]=g(E[X]).
\end{align*}

Next we show $E\left[ \Delta _{n}^{2}\right] =O\left( 2^{-(1+\alpha
)n}\right) $ for all $n\geq 0$. We pick $\delta >0$ small enough so that $%
g(\cdot )$ is continuously differentiable in a neighborhood of size $\delta $
around $\nu $ and the locally Holder continuous condition holds as well.
\begin{align*}
\left\vert \Delta _{n}\right\vert =& \left\vert \Delta _{n}\right\vert
I\left( \max \left( \left\Vert S_{2^{n}}^{O}/2^{n}-\nu \right\Vert
_{2},\left\Vert S_{2^{n}}^{E}/2^{n}-\nu \right\Vert _{2}\right) >\delta
/2\right) \\
& +\left\vert \Delta _{n}\right\vert I\left( \left\Vert
S_{2^{n}}^{O}/2^{n}-\nu \right\Vert _{2}\leq \delta /2,\left\Vert
S_{2^{n}}^{E}/2^{n}-\nu \right\Vert _{2}\leq \delta /2\right) \\
\leq & \left\vert \Delta _{n}\right\vert I\left( \left\Vert
S_{2^{n}}^{O}/2^{n}-\nu \right\Vert _{2}>\delta /2\right) +\left\vert \Delta
_{n}\right\vert I\left( \left\Vert S_{2^{n}}^{E}/2^{n}-\nu \right\Vert
_{2}>\delta /2\right) \\
& +\left\vert \Delta _{n}\right\vert I\left( \left\Vert
S_{2^{n}}^{O}/2^{n}-\nu \right\Vert _{2}\leq \delta /2,\left\Vert
S_{2^{n}}^{E}/2^{n}-\nu \right\Vert _{2}\leq \delta /2\right) .
\end{align*}%
When $\left\Vert S_{2^{n}}^{O}/2^{n}-\nu \right\Vert _{2}\leq \delta /2$ and
$\left\Vert S_{2^{n}}^{E}/2^{n}-\nu \right\Vert _{2}\leq \delta /2$, we have
$\left\Vert S_{2^{n+1}}/2^{n+1}-\nu \right\Vert _{2}\leq \delta $ and $%
\left\Vert S_{2^{n}}^{O}/2^{n}-S_{2^{n}}^{E}/2^{n}\right\Vert _{2}\leq
\delta $, thus 
\begin{align*}
\Delta _{n}& =\frac{1}{2}\left( g\left( S_{2^{n+1}}/2^{n+1}\right) -g\left(
S_{2^{n}}^{O}/2^{n}\right) \right) +\frac{1}{2}\left( g\left(
S_{2^{n+1}}/2^{n+1}\right) -g\left( S_{2^{n}}^{E}/2^{n}\right) \right) \\
& =\frac{1}{4}Dg\left( \xi _{n}^{O}\right) ^{T}\frac{%
S_{2^{n}}^{E}-S_{2^{n}}^{O}}{2^{n}}+\frac{1}{4}Dg\left( \xi _{n}^{E}\right)
^{T}\frac{S_{2^{n}}^{O}-S_{2^{n}}^{E}}{2^{n}} \\
& =\frac{1}{4}\left( Dg\left( \xi _{n}^{O}\right) -Dg\left( \xi
_{n}^{E}\right) \right) ^{T}\frac{S_{2^{n}}^{E}-S_{2^{n}}^{O}}{2^{n}},
\end{align*}%
where $\xi _{n}^{O}$ is some value between $S_{2^{n}}^{O}/2^{n}$ and $%
S_{2^{n+1}}/2^{n+1}$, and $\xi _{n}^{E}$ is some value between $%
S_{2^{n}}^{E}/2^{n}$ and $S_{2^{n+1}}/2^{n+1}$. It is not hard to see that
\begin{equation*}
\left\Vert \xi _{n}^{O}-\xi _{n}^{E}\right\Vert _{2}=\left\Vert \frac{%
U_{n}^{O}+U_{n}^{E}}{2}\cdot \left( \frac{S_{2^{n}}^{O}}{2^{n}}-\frac{%
S_{2^{n}}^{E}}{2^{n}}\right) \right\Vert _{2}\leq \left\Vert \frac{%
S_{2^{n}}^{E}}{2^{n}}-\frac{S_{2^{n}}^{O}}{2^{n}}\right\Vert _{2}.
\end{equation*}%
Hence, using the fact that $\kappa \left( \cdot \right) $ is bounded on
compact sets, we have that there exists a deterministic constant $c\in
\left( 0,\infty \right) $ (depending on $\delta $) such that%
\begin{align*}
& E\left( \left\vert \Delta _{n}\right\vert ^{2}I\left( \left\Vert
S_{2^{n}}^{O}/2^{n}-\nu \right\Vert _{2}\leq \delta /2,\left\Vert
S_{2^{n}}^{E}/2^{n}-\nu \right\Vert _{2}\leq \delta /2\right) \right) \\
\leq & \ cE\left( \left\Vert \frac{S_{2^{n}}^{O}-S_{2^{n}}^{E}}{2^{n}}%
\right\Vert _{2}^{2(1+\alpha )}\right) =O\left( 2^{-(1+\alpha )n}\right) ,
\end{align*}%
where the last estimate follows from \cite{Bahr:1965}.

On the other hand, in order to analyze,
\begin{equation}
E\left[ \lvert \Delta _{n}\rvert ^{2}I\left( \left\Vert
S_{2^{n}}^{O}/2^{n}-\nu \right\Vert _{2}>\delta /2\right) \right] .
\label{tail}
\end{equation}%
If we could assume that the $X_{i}$'s have a finite moment generating
function in a neighborhood of the origin it would be easy to see that (\ref%
{tail}) decays at a speed which is $o\left( 2^{-n\left( 1+\alpha \right)
}\right) $ (actually the rate would be superexponentially fast in $n$).
However, we are not assuming the existence of a finite moment generating
function, but we are assuming the existence of finite second moments. The
intuition that we will exploit is that the large deviations event that is
being introduced in (\ref{tail}) would be driven (in the worst case) by a
large jump (that is, we operate based on intuition borrowed from large
deviations theory for heavy-tailed increments). So, following this
intuition, we define, for some $\delta ^{\prime }>0$ small to be determined
in the sequel, the set
\begin{equation*}
\mathcal{A}_{n}=\{1\leq i\leq 2^{n}:\left\Vert X_{i}-v\right\Vert _{2}\geq
2^{n\left( 1-\delta ^{\prime }\right) }\}
\end{equation*}
and $N_{n}=\left\vert \mathcal{A}_{n}\right\vert $. In simple words, $N_{n}$
is the number of increments defining $S_{2^{n}}^{O}$ which are large. Note
that
\begin{eqnarray*}
&&E\left[ \lvert \Delta _{n}\rvert ^{2}I\left( \left\Vert
S_{2^{n}}^{O}/2^{n}-\nu \right\Vert _{2}>\delta /2\right) \right] \\
&=&E\left[ \lvert \Delta _{n}\rvert ^{2}I\left( \left\Vert
S_{2^{n}}^{O}/2^{n}-\nu \right\Vert _{2}>\delta /2\right) I\left(
N_{n}=0\right) \right] \\
&&+E\left[ \lvert \Delta _{n}\rvert ^{2}I\left( \left\Vert
S_{2^{n}}^{O}/2^{n}-\nu \right\Vert _{2}>\delta /2\right) I\left( N_{n}\geq
1\right) \right] .
\end{eqnarray*}%
We can easily verify using Chernoff's bound that for any $\gamma >0$, we
have that
\begin{equation*}
P\left( \left\Vert S_{2^{n}}^{O}/2^{n}-\nu \right\Vert _{2}>\delta
/2|N_{n}=0\right) =o\left( 2^{-n\gamma }\right) ,
\end{equation*}%
this implies that
\begin{eqnarray*}
&&E\left[ \lvert \Delta _{n}\rvert ^{2}I\left( \left\Vert
S_{2^{n}}^{O}/2^{n}-\nu \right\Vert _{2}>\delta /2\right) I\left(
N_{n}=0\right) \right] \\
&\leq &E\left[ \lvert \Delta _{n}\rvert ^{2\left( 1+\alpha \right) }\right]
^{1/(1+\alpha )}P\left( \left\Vert S_{2^{n}}^{O}/2^{n}-\nu \right\Vert
_{2}>\delta /2,N_{n}=0\right) ^{\alpha /\left( 1+\alpha \right) } \\
&=&o\left( 2^{-n\left( 1+\alpha \right) }\right) .
\end{eqnarray*}%
On the other hand, note that
\begin{eqnarray*}
&&2^{-2n}E\left[ \left\Vert X_{i}-v\right\Vert _{2}^{2}I\left( \left\Vert
X_{i}-v\right\Vert _{2}>2^{n\left( 1-\delta ^{\prime }\right) \left(
1+\alpha \right) }\right) \right] \\
&=&2^{-2n+1}\int_{2^{n\left( 1-\delta ^{\prime }\right) \left( 1+\alpha
\right) }}^{\infty }tP\left( \left\Vert X_{i}-v\right\Vert _{2}>t\right) dt
\\
&\leq &2^{-2n+1}\int_{2^{n\left( 1-\delta ^{\prime }\right) \left( 1+\alpha
\right) }}^{\infty }\frac{E\left( \left\Vert X_{i}-v\right\Vert
_{2}^{3\left( 1+\alpha \right) }\right) }{t^{2+3\alpha }}dt \\
&=&O\left( 2^{-2n-n\left( 1+3\alpha \right) \left( 1-\delta ^{\prime
}\right) \left( 1+\alpha \right) }\right) .
\end{eqnarray*}%
Using the previous estimate, it follows easily that
\begin{equation*}
E\left[ \lvert \Delta _{n}\rvert ^{2}I\left( N_{n}=1\right) \right] =O\left(
2^{n}\cdot 2^{-2n-n\left( 1+3\alpha \right) \left( 1-\delta ^{\prime
}\right) \left( 1+\alpha \right) }\right) .
\end{equation*}%
The previous expression is $O\left( 2^{-2n}\right) $ if $\delta ^{\prime }>0$
is chosen sufficiently small. Similarly, for any fixed $k$,
\begin{eqnarray*}
&&E\left[ \lvert \Delta _{n}\rvert ^{2}I\left( N_{n}=k\right) \right] \\
&=&O\left( 2^{\left( k-2\right) n}2^{-n\cdot k\left( \left( 1+3\alpha
\right) \left( 1-\delta ^{\prime }\right) \left( 1+\alpha \right) \right)
}\right) =O\left( 2^{-2n}\right) .
\end{eqnarray*}%
On the other hand,
\begin{equation*}
P\left( N_{n}\geq m\right) =O\left( 2^{nm}2^{-nm\cdot 3\left( 1-\delta
^{\prime }\right) \left( 1+\alpha \right) }\right) .
\end{equation*}%
We then obtain that by selecting $\delta ^{\prime }>0$ sufficiently small
and $m$ large so that
\begin{equation*}
m\cdot \alpha \left( 3\left( 1-\delta ^{\prime }\right) -\frac{1}{1+\alpha }%
\right) \geq 2,
\end{equation*}%
we conclude
\begin{equation*}
E\left[ \lvert \Delta _{n}\rvert ^{2}I\left( N_{n}\geq m\right) \right] \leq
E\left[ \lvert \Delta _{n}\rvert ^{2\left( 1+\alpha \right) }\right]
^{1/\left( 1+\alpha \right) }P\left( N_{n}\geq m\right) ^{\alpha /\left(
1+\alpha \right) }=O\left( 2^{-2n}\right) .
\end{equation*}%
Consequently, we have that
\begin{eqnarray*}
&&E\left[ \lvert \Delta _{n}\rvert ^{2}I\left( \left\Vert
S_{2^{n}}^{O}/2^{n}-\nu \right\Vert _{2}>\delta /2\right) I\left( N_{n}\geq
1\right) \right] \\
&\leq &\sum_{k=1}^{m-1}E\left[ \lvert \Delta _{n}\rvert ^{2}I\left(
N_{n}=k\right) \right] +E\left[ \lvert \Delta _{n}\rvert ^{2}I\left(
N_{n}\geq m\right) \right] =O\left( 2^{-2n}\right) .
\end{eqnarray*}

A similar analysis yields that
\begin{equation*}
E\left[ \left\vert \Delta _{n}\right\vert ^{2}I\left( \max \left( \left\Vert
S_{2^{n}}^{O}/2^{n}-\nu \right\Vert _{2},\left\Vert S_{2^{n}}^{E}/2^{n}-\nu
\right\Vert _{2}\right) >\delta /2\right) \right] =O\left( 2^{-2n}\right) ,
\end{equation*}%
therefore the estimator $Z$ has finite variance.

Finally the sampling complexity of producing one copy of $\Delta_n$ is
\begin{equation*}
C_n=2^{n+1}+c=O\left(2^n\right)
\end{equation*}
with some constant $c>0$.
\end{proof}

\subsection{Application to steady-state regenerative simulation}

The context of steady-state simulation provides an important instance in
which developing unbiased estimators is desirable. Recall that if $%
\left(W(n):n\ge0\right)$ is a positive recurrent regenerative process taking
values on some space $\mathcal{Y}$, then for all measurable set $A$, we have
the following limit holds with probability one
\begin{equation*}
\pi(A):=\lim_{m\to\infty}\frac{1}{m}\sum_{n=0}^mI\left(W(n)\in A\right)=%
\frac{E_0\left[\sum_{n=0}^{\tau-1}I\left(W(n)\in A\right)\right]}{E_0\left[%
\tau\right]},
\end{equation*}
where the notation $E_0$ indicates that $W(\cdot)$ is zero-delayed under the
associated probability measure $P_0\left(\cdot\right)$. The limiting measure
$\pi(\cdot)$ is the unique stationary distribution of the process $W(\cdot)$%
; for additional discussion on regenerative processes see the appendix on
regenerative process in \cite{Asmussen-Glynn-2008}, and also \cite%
{Asmussen-2000}. Most ergodic Markov chain that arise in practice are
regenerative; certainly all irreducible and positive recurrent countable
state-space Markov chains are regenerative.

A canonical example which is useful to keep in mind to conceptualize a
regenerative process is the waiting time sequence of the single server
queue. In which case, it is well known that the waiting time of the $n$-th
customer, $W(n)$, satisfies the recursion $W(n+1)=\max \left(
W(n)+Y(n+1),0\right) $, where the $Y(n)$'s form an iid sequence of random
variables with negative mean. The waiting time sequence regenerates at zero,
so if $W(0)=0$, the waiting time sequence forms a zero-delayed regenerative
process. Let $f(\cdot )$ be a bounded measurable function and write
\begin{equation*}
X_{1}=\sum_{n=1}^{\tau -1}f\left( W(n)\right) ~~~\mbox{and}~~~X_{2}=\tau ,
\end{equation*}%
then we can estimate the stationary expectation $E_{\pi }f\left( W\right) $
via the ratio
\begin{equation}
E_{\pi }\left[ f\left( W\right) \right] =\frac{E_{0}\left[ X_{1}\right] }{%
E_{0}\left[ X_{2}\right] }.  \label{ratio-estimator}
\end{equation}%
Since $\tau \geq 1$, it follows that for $g\left( x_{1},x_{2}\right)
=x_{1}/x_{2}$, assumptions can be easily verified and therefore Theorem \ref%
{thm-function-expectation} applies, we need to assume that $E\left(
\left\vert X_{1}\right\vert ^{3+\varepsilon }\right) <\infty $\ and %
$E\left( \tau ^{3+\varepsilon }\right) <\infty $\ for some $%
\varepsilon >0$.

\subsection{Additional applications}

In addition to steady-state simulation, ratio estimators such (\ref%
{ratio-estimator}) arise in the context of particle filters and
state-dependent importance sampling for Bayesian computations, see \cite%
{DelMoral-2004} and \cite{Liu-2008}.

In the context of Bayesian inference, one is interested in estimating
expectations from some density $(\pi(y):y\in\mathcal{Y})$ fo the form $%
\pi(y)=h(y)/\gamma$, where $h(\cdot)$ is a non-negative function with a
given (computable) functional form and $\gamma>0$ is a normalizing constant
which is not computable, but is well defined (i.e. finite) and ensures that $%
\pi(\cdot)$ is indeed a well defined density on $\mathcal{Y}$. Since $%
\gamma>0$ is unknown one must resort to techniques such as Markov chain
Monte Carlo or sequential importance sampling to estimate $E_{\pi}\left[f(Y)%
\right]$ (for any integrable function $f(\cdot)$), see for instance \cite%
{Liu-2008}.

Ultimately, the use of sequential importance samplers or particle filters
relies on the identity
\begin{equation}  \label{eq-seq-sample}
E_{\pi}\left[f(Y)\right]=E_{q}\left[\frac{h(Y)}{q(Y)}f(Y)\right]/E_{q}\left[%
\frac{h(Y)}{q(Y)}\right],
\end{equation}
where $(q(y):y\in\mathcal{Y})$ is a density on $\mathcal{Y}$ and $E_q\left[%
\cdot\right]$ denotes the expectation operator associated to $q(\cdot)$ (and
we use $P_q(\cdot)$ for the associated probability). Of course, we must have
that the likelihood ratio $\pi(Y)/q(Y)$ well defined almost surely with
respect to $P_q(\cdot)$ and
\begin{equation*}
E_q\left[\frac{\pi(Y)}{q(Y)}\right]=1.
\end{equation*}

Thus, by using sequential importance sampling or particle filters one
produces a ratio estimator (\ref{eq-seq-sample}) and therefore the
application of our result in this setting is very similar to the one
described in the previous subsection. The verification of Theorem \ref%
{thm-function-expectation} requires additional assumption on the selection
of $q(\cdot)$, which should have heavier tails than $\pi(\cdot)$ in order to
satisfy Assumption 3.

\section{Stochastic Convex Optimization}

\label{sec:convex}

In this section we study a wide range of stochastic optimization problems
and we show that the general principle applies. This section studies
situations in which, going back to Section \ref{sec:general-principle}, the
derivative $T_{\mu }^{\theta }$ may be difficult to characterize and
analyze, but the general principle is still applicable. So, in this section
we study its applications directly.

Consider the following constrained stochastic convex optimization problem
\begin{equation}
\begin{matrix}
\displaystyle\min & f(\beta )=E_{\mu }\left[ F(\beta ,X)\right] \\
\text{s.t.} & G(\beta )\leq 0,%
\end{matrix}
\label{nice-cvx}
\end{equation}%
where $\mathcal{D}=\{\beta \in \mathbb{R}^{d}:F\left( \beta \right) \leq 0\}$
is a nonempty closed subset of $\mathbb{R}^{d}$. $f$ is a convex map from $%
\mathbb{R}^{d}$ to $\mathbb{R}$. $G\left( \beta \right) =\left( g_{1}(\beta
),\ldots ,g_{m}(\beta )\right) ^{T}$ is a vector-valued convex function for
some $m\in \mathbb{N}$. $X$ is a random vector whose probability
distribution $\mu $ is supported on a set $\Omega \subset \mathbb{R}^{k}$,
and $F:\mathcal{D}\times \Omega \rightarrow \mathbb{R}$.

Let $\beta _{\ast }$ denote the optimal solution and $f_{\ast }=f(\beta
_{\ast })$ denote the optimal objective value. Lagrangian of problem (\ref%
{nice-cvx}) is
\begin{equation}
L\left( \beta ,\lambda \right) =f\left( \beta \right) +\lambda ^{T}G\left(
\beta \right) .  \label{lagrangian}
\end{equation}%
If $f(\cdot )$ and $g_{i}(\cdot )$'s are continuously differentiable for $%
i=1,\ldots ,m$, the following \textit{Karush-Kuhn-Tucker} (KKT) conditions
are sufficient and necessary for optimality:
\begin{align}
\nabla _{\beta }L\left( \beta _{\ast },\lambda _{\ast }\right) & =\nabla
f(\beta _{\ast })+\nabla G\left( \beta _{\ast }\right) \lambda _{\ast }=0,
\label{kkt-1} \\
G\left( \beta _{\ast }\right) & \leq 0,  \label{kkt-2} \\
\lambda _{\ast }^{T}G\left( \beta _{\ast }\right) & =0,  \label{kkt-3} \\
\lambda _{\ast }& \geq 0,  \label{kkt-4}
\end{align}%
where $\lambda _{\ast }\in \mathbb{R}^{m}$ is the Lagrangian multiplier
corresponding to $\beta _{\ast }$.

One of the standard tools in such settings is the method of Sample Average
Approximation (SAA), which consists in replacing the expectations by the
empirical means. Suppose we have $n$ iid copies of the random vector $X$,
denoted as $\{X_1,\ldots,X_n\}$, we solve the following optimization problem
\begin{equation}  \label{saa-cvx}
\begin{matrix}
\displaystyle \min & f_n(\beta)=\frac{1}{n}\sum_{i=1}^nF(\beta,X_i) \\
\text{s.t.} & G(\beta) \le0%
\end{matrix}%
\end{equation}
as an approximation to the original problem (\ref{nice-cvx}). Let $\beta_n$
denote the optimal solution and let $\hat{f}_n=f_n(\beta_n)$ denote the
optimal value of the SAA problem (\ref{saa-cvx}). The traditional SAA
approach is to use them as estimators to the true optimal solution $\beta_*$
and optimal target value $f_*$ of the problem (\ref{nice-cvx}). Although the
SAA estimators are easy to construct and consistent, they are biased.
Proposition 5.6 in \cite{Shapiro:2009} shows $E[\hat{f}_n]\le f_*$ for any $%
n\in\mathbb{N}$.

We construct unbiased estimators for the optimal solution and optimal value
of problem (\ref{nice-cvx}) by utilizing the SAA estimators. Let $%
\beta_{2^{n+1}}$, $\beta^O_{2^n}$, $\beta^E_{2^n}$ denote the SAA optimal
solutions as
\begin{align*}
\beta_{2^{n+1}}&=\argmin_{G(\beta)\le0}f_{2^{n+1}}(\beta)=\argmin%
_{G(\beta)\le0}\frac{1}{2^{n+1}}\sum_{i=1}^{2^{n+1}}F\left(\beta,X_i\right),
\\
\beta^O_{2^n} &=\argmin_{G(\beta)\le0}f^O_{2^n}(\beta)=\argmin_{G(\beta)\le0}%
\frac{1}{2^n}\sum_{i=1}^{2^n}F\left(\beta,X_i^O\right), \\
\beta^E_{2^n} &= \argmin_{G(\beta)\le0}f^E_{2^n}(\beta)=\argmin%
_{G(\beta)\le0}\frac{1}{2^n}\sum_{i=1}^{2^n}F\left(\beta,X_i^E\right).
\end{align*}
Let $\hat{f}_{2^{n+1}} = f_{2^{n+1}}\left(\beta_{2^{n+1}}\right)$, $\hat{f}%
_{2^n}^O=f^O_{2^n}\left(\beta^O_{2^n}\right)$ and $\hat{f}%
_{2^n}^E=f^E_{2^n}\left(\beta^E_{2^n}\right)$ denote the SAA optimal values.
Similarly we let $\lambda_{2^{n+1}}$, $\lambda_{2^n}^O$ and $\lambda_{2^n}^E$
denote the corresponding Lagrange multipliers.

We define
\begin{equation*}
\Delta_n=\hat{f}_{2^{n+1}}-\frac{1}{2}\left(\hat{f}^O_{2^n}+\hat{f}%
^E_{2^n}\right) \ \mbox{ and } \ \bar{\Delta}_n=\beta_{2^{n+1}}-\frac{1}{2}%
\left(\beta_{2^n}^O+\beta_{2^n}^E\right)
\end{equation*}
for all $n\ge0$, then the estimator of the optimal value $f_*$ is
\begin{equation}  \label{cvx-value-estimator}
Z=\frac{\Delta_N}{p(N)}+\hat{f}_1
\end{equation}
and the estimator of the optimal solution $\beta_*$ is
\begin{equation}  \label{cvx-solution-estimator}
\bar{Z}=\frac{\bar{\Delta}_n}{p(N)}+\beta_1,
\end{equation}
where $N$ was defined in Section \ref{sec:general-principle}. We now impose
assumptions in this setting so that Assumption \ref%
{general-principle-assumptions} for the general principles can be verified
for both $\Delta_n$ and $\bar{\Delta}_n$.

\begin{asu}
\label{convex-assumptions} Stochastic convex optimization assumptions:

\refstepcounter{subassumption} \thesubassumption~\ignorespaces\label%
{asu-cvx-compact} The feasible region $\mathcal{D}\subset \mathbb{R}^{d}$ is
compact.

\refstepcounter{subassumption} \thesubassumption~\ignorespaces\label%
{asu-cvx-uniqueness} $f$ has a unique optimal solution $\beta _{\ast }\in
\mathcal{D}$.

\refstepcounter{subassumption} \thesubassumption~\ignorespaces\label%
{asu-cvx-convex} $F(\cdot ,X)$ is finite, convex and twice continuously
differentiable on $\mathcal{D}$ a.s.

\refstepcounter{subassumption} \thesubassumption~\ignorespaces\label%
{asu-cvx-hcalm} There exists a locally bounded measurable function $\kappa
:\Omega \rightarrow \mathbb{R}_{+}$, $\gamma >0$ and $\delta >0$ such that
\begin{equation*}
\lvert F\left( \beta ^{\prime },X\right) -F\left( \beta ,X\right) \rvert
\leq \kappa (X)\lVert \beta ^{\prime }-\beta \rVert ^{\gamma }
\end{equation*}%
for all $\beta ,\beta ^{\prime }\in \mathcal{D}$ with $\lVert \beta ^{\prime
}-\beta \rVert \leq \delta $ and $X\in \Omega $; and $\kappa (X)$ has finite
moment generating function in a neighborhood of the origin.

\refstepcounter{subassumption} \thesubassumption~\ignorespaces\label%
{asu-cvx-mgf} Define,%
\begin{equation}
M_{\beta }(t)=E\left[ \exp \left( t\left( F(\beta ,X_{i})-f(\beta )\right)
\right) \right]  \label{eq:mgf}
\end{equation}%
and assume that there exists $\delta _{0}>0$ and $\sigma ^{2}>0$ such that
for $\left\vert t\right\vert \leq \delta _{0}$,
\begin{equation*}
\sup_{\beta \in \mathcal{D}}M_{\beta }(t)\leq \exp \left( \sigma
^{2}t^{2}/2\right) .
\end{equation*}

\refstepcounter{subassumption} \thesubassumption~\ignorespaces\label%
{asu-cvx-gradient} There is $\delta _{0}^{\prime }>0$ and $t>0$ such that%
\begin{equation*}
\sup_{\left\Vert \beta -\beta _{\ast }\right\Vert \leq \delta _{0}^{\prime
}}E\left[ \exp \left( t\left\Vert \nabla _{\beta }F(\beta ,X)\right\Vert
\right) \right] <\infty .
\end{equation*}

\refstepcounter{subassumption} \thesubassumption~\ignorespaces\label%
{asu-cvx-hessian} $E\left[ \left\Vert \nabla _{\beta \beta }^{2}F\left(
\beta _{\ast },X\right) \right\Vert ^{p}\right] <\infty $ with some $p>2$.

\refstepcounter{subassumption} \thesubassumption~\ignorespaces\label%
{asu-cvx-constraint-convex} $G(\beta )=\left( g_{1}(\beta ),\ldots
,g_{m}(\beta )\right) ^{T}$ and $g_{i}(\cdot )$ is twice continuously
differentiable convex function for all $1\leq i\leq m$.

\refstepcounter{subassumption} \thesubassumption~\ignorespaces\label%
{asu-cvx-slater} There is $\beta \in \mathcal{D}$ such that $G\left( \beta
\right) <0$ (Slater conditions ensures strong duality).

\refstepcounter{subassumption} \thesubassumption~\ignorespaces\label%
{asu-cvx-licq} LICQ holds at $\beta _{\ast }$, i.e., the gradient vectors $%
\left\{ \nabla g_{i}(\beta _{\ast }):g_{i}(\beta _{\ast })=0\right\} $ are
linearly independent (LICQ is the weakest condition to ensure the uniqueness
of Lagrangian multiplier; see \cite{Wachsmuth2013} for instance).

\refstepcounter{subassumption} \thesubassumption~\ignorespaces\label%
{asu-cvx-strict-complementary} Strict complementarity condition holds, i.e.,
$\lambda _{\ast }(i)>0$ when $g_{i}(\beta _{\ast })=0$ for all $i=1,\ldots
,m $,


\end{asu}

We summarize the discussion of unbiased estimator for the optimal solution $%
\beta_*$ as Theorem \ref{unbiased-optimizer} in Section \ref%
{unbiased-solution}, and unbiased estimator for the optimal objective value $%
f_*$ as Theorem \ref{unbiased-optimal} in Section \ref{unbiased-value}.

\subsection{Unbiased estimator of optimal solution}

\label{unbiased-solution}

In this section, we will utilize the large deviation principles for the SAA
optimal solutions develop in \cite{Xu:2010}. We first provide the following
lemma to summarize the LDP.

\begin{lemma}
\label{ldp-lemma} If Assumptions \ref{asu-cvx-compact}, \ref%
{asu-cvx-uniqueness}, \ref{asu-cvx-hcalm}, \ref%
{asu-cvx-mgf} hold, then for every $\epsilon>0$, there exist positive
constants $c(\epsilon)$ and $\alpha(\epsilon)$, independent of $n$, such
that for $n$ sufficiently large
\begin{equation*}
P\left(\lVert\beta_{2^n}-\beta_*\rVert\ge\epsilon\right)\le
c_{\epsilon}e^{-2^n\alpha(\epsilon)},
\end{equation*}
where $\alpha(\epsilon)$ is locally quadratic at the origin, i.e., $%
\alpha(\epsilon)=\alpha_0\epsilon^2$ as $\epsilon\to0$ with $\alpha_0>0$.
\end{lemma}

\begin{proof}
For $\beta \in \mathcal{D},$ let $I_{\beta }(z)=\sup_{t\in \mathbb{R}%
}\{zt-\log M_{\beta }(t)\}$, where $M_{\beta }(t)$ is as defined in (\ref%
{eq:mgf}). The proof of Theorem 4.1 in \cite{Xu:2010} has that if the
assumptions required in Lemma \ref{ldp-lemma} are all enforced, then
\begin{eqnarray*}
&&P\left( \lVert \beta _{2^{n}}-\beta _{\ast }\rVert \geq \epsilon \right) \\
&\leq &\exp \left( -2^{n}\lambda \right) +\sum_{i=1}^{M}\exp \left(
-2^{n}\min (I_{\bar{\beta}_{i}}(\epsilon /4),I_{\bar{\beta}_{i}}(-\epsilon
/4))\right) ,
\end{eqnarray*}%
where $\lambda >0$, $\{\bar{\beta}_{i}\in \mathcal{D}:1\leq i\leq M\}$ is a $%
v$-net constructed by the finite covering theorem, i.e., there exits $v>0$
such that for every $\beta \in \mathcal{D}$, there exists $\bar{\beta}_{i}$,
$i\in \{1,\ldots ,M\}$, $\lVert \beta -\bar{\beta}_{i}\rVert \leq v$,
\begin{equation*}
\lvert F(\beta ,X)-F(\bar{\beta}_{i},X)\rvert \leq \kappa (X)\lVert \beta -%
\bar{\beta}_{i}\rVert ^{\gamma }\ \mbox{  and  }\ \lvert f(\beta )-f(\bar{%
\beta}_{i})\rvert \leq \epsilon /4,
\end{equation*}%
and
\begin{equation*}
\min (I_{\bar{\beta}_{i}}(\epsilon /4),I_{\bar{\beta}_{i}}(-\epsilon
/4))\geq \frac{\epsilon ^{2}}{32\sigma ^{2}},
\end{equation*}%
by Assumption \ref{asu-cvx-hcalm} and Remark 3.1 in \cite{Xu:2010}. Since
the size of $v$-net, $M$, grows in polynomial order of $\epsilon $, we
complete the proof.
\end{proof}

\begin{theorem}
\label{unbiased-optimizer} If Assumptions \ref{convex-assumptions} is in
force, then $E\left[\bar{Z}\right]=\beta_*$, $Var(\bar{Z})<\infty$ and the
computation complexity required to produce $\bar{Z}$ is bounded in
expectation.
\end{theorem}

\begin{proof}
If Assumptions \ref{asu-cvx-compact}, \ref{asu-cvx-convex}, \ref%
{asu-cvx-constraint-convex}, \ref{asu-cvx-slater}, \ref{asu-cvx-uniqueness}, %
\ref{asu-cvx-licq} and \ref{asu-cvx-strict-complementary} hold, the
following result is given on page 171 of \cite{Shapiro:2009}
\begin{equation}
2^{n/2}%
\begin{bmatrix}
\beta _{2^{n}}-\beta _{\ast } \\
\lambda _{2^{n}}-\lambda _{\ast }%
\end{bmatrix}%
\Longrightarrow \mathcal{N}\left( 0,J^{-1}\Gamma J\right) ,  \label{clt}
\end{equation}%
where
\begin{equation*}
J=%
\begin{bmatrix}
H & A \\
A^{T} & 0%
\end{bmatrix}%
\ \mbox{  and  }\ \Gamma =%
\begin{bmatrix}
\Sigma & 0 \\
0 & 0%
\end{bmatrix}%
,
\end{equation*}%
$H=\nabla _{\beta \beta }^{2}L\left( \beta _{\ast },\lambda _{\ast }\right)
\in \mathbb{R}_{d\times d}$, $A$ is the matrix whose columns are formed by
vectors $\nabla g_{i}(\beta _{\ast })$ when $g_{i}(\beta _{\ast })=0$ for $%
i=1,\ldots ,m$, and $\Sigma =E\left[ \left( \nabla F(\beta _{\ast
},X)-\nabla f(\beta _{\ast })\right) \left( \nabla F(\beta _{\ast
},X)-\nabla f(\beta _{\ast })\right) ^{T}\right] $. Nonsingularity of $J$ is
guaranteed by Assumptions \ref{asu-cvx-licq} and \ref%
{asu-cvx-strict-complementary}.

We first show $\bar{Z}$ is unbiased. For $n\ge0$,
\begin{equation*}
E\left[\bar{\Delta}_n\right]=E\left[\beta_{2^{n+1}}\right]-E\left[\beta_{2^n}%
\right].
\end{equation*}
Since the feasible region $\mathcal{D}\subset\mathbb{R}^d$ is closed and
bounded by Assumption \ref{asu-cvx-compact}, $\left\{\beta_{2^n}:n\ge0\right%
\}$ is uniformly integrable. With $\beta_{2^n}\to\beta_*$ in (\ref{clt}), we
have
\begin{equation*}
E\left[\bar{Z}\right]=\sum_{n=1}^{\infty}E\left[\bar{\Delta}_n\right]+E\left[%
\beta_1\right]=\lim_{n\to\infty}E\left[\beta_{2^n}\right]=\beta_*.
\end{equation*}

We next prove $Var(\bar{Z})<\infty $ by showing $E\left[ \bar{\Delta}_{n}%
\bar{\Delta}_{n}^{T}\right] =O\left( 2^{-(1+\alpha )n}\right) $ with some $%
\alpha >0$. Let $m\left( \beta _{\ast }\right) =E\left[ \nabla _{\beta \beta
}^{2}F\left( \beta _{\ast },X\right) \right] $. The key ingredients are --
firstly we use the large deviation principle (LDP) of $\{\beta
_{2^{n}}:n\geq 0\}$ to get moderate deviation estimates for $\{\beta
_{2^{n}}:n\geq 0\}$, secondly to use extended contraction principle with
modified optimization problems to translate the LDP to the sequence of
Lagrange multipliers $\{\lambda _{2^{n}}:n\geq 0\}$. For the first part, we
have by Lemma \ref{ldp-lemma} that
\begin{equation*}
P\left( \lVert \beta _{2^{n}}-\beta _{\ast }\rVert \geq \epsilon \right)
=\exp \left( -2^{n}\alpha (\epsilon )+o(2^{n})\right)
\end{equation*}%
for all $\epsilon >0$ sufficiently small and $\alpha (\epsilon )=\alpha
_{0}\epsilon ^{2}(1+o(1))$ as $\epsilon \rightarrow 0$ for $\alpha _{0}>0$.
This yields moderate deviation estimates for $\{\beta _{2^{n}}:n\geq 0\}$.
In particular, we let $\epsilon \rightarrow 0$ at a speed of the the form $%
\epsilon =2^{-\rho n}$ for $1/4<\rho <1/2$ and the limit above will still
provide the correct rate of convergence, i.e.,
\begin{equation*}
P\left( \lVert \beta _{2^{n}}-\beta _{\ast }\rVert \geq 2^{-\rho n}\right)
=\exp \left( -\alpha _{0}2^{(1-2\rho )n}+o\left( 2^{(1-2\rho )n}\right)
\right) .
\end{equation*}%
Then to translate this LDP to $\{\lambda _{2^{n}}:n\geq 0\}$, we consider a
family of modified optimization problems (indexed by $\eta $)
\begin{equation}
\begin{matrix}
\displaystyle\min & f_{\eta }(\beta )=E_{\mu }\left[ F(\beta ,X)\right]
+\eta ^{T}\beta \\
\text{s.t.} & G(\beta )\leq 0%
\end{matrix}%
,  \label{modified-opt}
\end{equation}%
and its associated optimal solution, $\beta (\eta )$, with the associated
Lagrange multiplier, $\lambda (\eta )$, for the modified problem. It follows
that $\lambda (\cdot )$ is continuously differentiable as a function of $%
\eta $ in a neighborhood of the origin, this is a consequence of Assumptions %
\ref{asu-cvx-licq} and \ref{asu-cvx-strict-complementary}. Both $\beta (\eta
)$ and $\lambda (\eta )$ are characterized by the following KKT conditions
\begin{align}
\nabla f(\beta \left( \eta \right) )+\nabla G\left( \beta \left( \eta
\right) \right) \lambda \left( \eta \right) & =-\eta ,  \label{Akkt1} \\
G\left( \beta \left( \eta \right) \right) & \leq 0,  \label{Akkt2} \\
\lambda \left( \eta \right) ^{T}G\left( \beta \left( \eta \right) \right) &
=0,  \label{Akkt3} \\
\lambda \left( \eta \right) & \geq 0.  \label{Akkt4}
\end{align}%
By one of the KKT optimality conditions specified in (\ref{kkt-1}) for the
SAA problem we have that
\begin{equation}
0=\frac{1}{2^{n}}\sum_{i=1}^{2^{n}}\nabla _{\beta }F\left( \beta
_{2^{n}},X_{i}\right) +\nabla G\left( \beta _{2^{n}}\right) \lambda _{2^{n}}.
\label{Emp_KKT}
\end{equation}%
The previous equality implies that
\begin{equation*}
\nabla _{\beta }f\left( \beta _{2^{n}}\right) +\nabla _{\beta }G\left( \beta
_{2^{n}}\right) \cdot \lambda _{2^{n}}=-\bar{\eta}_{2^{n}},
\end{equation*}%
where
\begin{equation*}
\bar{\eta}_{2^{n}}=\frac{1}{2^{n}}\sum_{i=1}^{2^{n}}\left( \nabla _{\beta
}F\left( \beta _{2^{n}},X_{i}\right) -\nabla _{\beta }f\left( \beta
_{2^{n}}\right) \right) .
\end{equation*}%
Written in this form, we can identify that $\beta _{2^{n}}=\beta \left( \bar{%
\eta}_{2^{n}}\right) $ and $\lambda _{2^{n}}=\lambda \left( \bar{\eta}%
_{2^{n}}\right) $. We already know that $\left\{ \beta _{2^{n}}:n\geq
0\right\} $ has a large deviations principle, so the LDP can be derived for $%
\left\{ \bar{\eta}_{2^{n}}:n\geq 0\right\} $ by Theorem 2.1 of \cite%
{GaoZhao:2011} with Assumption \ref{asu-cvx-gradient}. Furthermore, the LDP
can then be derived for the Lagrange multipliers $\{\lambda _{2^{n}}=\lambda
(\bar{\eta}_{2^{n}}):n\geq 0\}$ because $\lambda (\cdot )$ is continuously
differentiable as a function of $\eta $ in a neighborhood of the origin, as
we mensioned earlier.

Then, from (\ref{Emp_KKT}) it follows by Taylor expansion that
\begin{align}
0=& \frac{1}{2^{n}}\sum_{i=1}^{2^{n}}\nabla _{\beta }F\left( \beta
_{2^{n}},X_{i}\right) +\nabla G\left( \beta _{2^{n}}\right) \lambda _{2^{n}}
\notag  \label{eq-optimizer} \\
=& \frac{1}{2^{n}}\sum_{i=}^{2^{n}}\nabla _{\beta }F\left( \beta _{\ast
},X_{i}\right) +\frac{1}{2^{n}}\sum_{i=1}^{2^{n}}\left( \nabla _{\beta
}F\left( \beta _{2^{n}},X_{i}\right) -\nabla _{\beta }F\left( \beta _{\ast
},X_{i}\right) \right) +\nabla G\left( \beta _{2^{n}}\right) \lambda _{2^{n}}
\notag \\
=& \frac{1}{2^{n}}\sum_{i=1}^{2^{n}}\nabla _{\beta }F\left( \beta _{\ast
},X_{i}\right) +\nabla G(\beta _{\ast })\lambda _{\ast }+\left( \frac{1}{%
2^{n}}\sum_{i=1}^{2^{n}}\nabla _{\beta \beta }^{2}F\left( \beta _{\ast
},X_{i}\right) -m(\beta _{\ast })\right) \cdot \left( \beta _{2^{n}}-\beta
_{\ast }\right)  \notag \\
& +\left( m(\beta _{\ast })+\lambda _{\ast }^{T}\nabla ^{2}G(\beta _{\ast
})\right) \cdot \left( \beta _{2^{n}}-\beta _{\ast }\right) +\nabla G(\beta
_{\ast })\left( \lambda _{2^{n}}-\lambda _{\ast }\right)  \notag \\
& +\bar{R}_{n,(\beta ,\beta )}+\bar{R}_{n,(\lambda ,\lambda )}+\bar{R}%
_{n,(\beta ,\lambda )},
\end{align}%
where
\begin{align*}
\bar{R}_{n,(\beta ,\beta )}& =O(\lVert \beta _{2^{n}}-\beta _{\ast }\rVert
^{2}), \\
\bar{R}_{n,(\lambda ,\lambda )}& =O\left( \lVert \lambda _{2^{n}}-\lambda
_{\ast }\rVert ^{2}\right) , \\
\bar{R}_{n,(\beta ,\lambda )}& =O\left( \lVert \beta _{2^{n}}-\beta _{\ast
}\rVert \lVert \lambda _{2^{n}}-\lambda _{\ast }\rVert \right) .
\end{align*}%
Let
\begin{equation*}
\bar{R}_{n}=\left( \frac{1}{2^{n}}\sum_{i=1}^{2^{n}}\nabla _{\beta \beta
}^{2}F\left( \beta _{\ast },X_{i}\right) -m(\beta _{\ast })\right) \cdot
\left( \beta _{2^{n}}-\beta _{\ast }\right)
\end{equation*}%
and let $\Lambda _{1}=m(\beta _{\ast })+\lambda _{\ast }T\nabla ^{2}G(\beta
_{\ast })\in \mathbb{R}_{d\times d}$, $\Lambda _{2}=\nabla G\left( \beta
_{\ast }\right) \in \mathbb{R}_{d\times m}$. Then we can rewrite (\ref%
{eq-optimizer}) as
\begin{equation*}
\Lambda _{1}(\beta _{2^{n}}-\beta _{\ast })+\Lambda _{2}(\lambda
_{2^{n}}-\lambda _{\ast })=-\left( \frac{1}{2^{n}}\sum_{i=1}^{2^{n}}\nabla
_{\beta }F\left( \beta _{\ast },X_{i}\right) +\nabla G\left( \beta _{\ast
}\right) \lambda _{\ast }+\bar{R}_{n}+\bar{R}_{n,(\beta ,\beta )}+\bar{R}%
_{n,(\lambda ,\lambda )}+\bar{R}_{n,(\beta ,\lambda )}\right) .
\end{equation*}%
Note that by Holder's inequality,
\begin{align}
E\left[ \bar{R}_{n}\bar{R}_{n}^{T}\right] & \leq E\left[ \left\Vert \frac{1}{%
2^{n}}\sum_{i=1}^{2^{n}}\nabla _{\beta \beta }^{2}F\left( \beta _{\ast
},X_{i}\right) -m(\beta _{\ast })\right\Vert ^{2}\cdot \left\Vert \beta
_{2^{n+1}}-\beta _{\ast }\right\Vert ^{2}\right]  \notag  \label{holder-1} \\
& \leq E\left[ \left\Vert \frac{1}{2^{n}}\sum_{i=1}^{2^{n}}\nabla _{\beta
\beta }^{2}F\left( \beta _{\ast },X_{i}\right) -m(\beta _{\ast })\right\Vert
^{p}\right] ^{2/p}E\left[ \lVert \beta _{2^{n+1}}-\beta _{\ast }\rVert
^{2p/(p-2)}\right] ^{(p-2)/p},
\end{align}%
where
\begin{equation}
E\left[ \left\Vert \frac{1}{2^{n}}\sum_{i=1}^{2^{n}}\nabla _{\beta \beta
}^{2}F\left( \beta _{\ast },X_{i}\right) -m(\beta _{\ast })\right\Vert ^{p}%
\right] =O\left( 2^{-np/2}\right)  \label{holder-2}
\end{equation}%
by \cite{Bahr:1965}, and by using the moderate large deviation estimate
\begin{align}
& E\left[ \lVert \beta _{2^{n+1}}-\beta _{\ast }\rVert ^{2p/(p-2)}\right]
\notag  \label{holder-3} \\
=& E\left[ \lVert \beta _{2^{n+1}}-\beta _{\ast }\rVert ^{2p/(p-2)}I\left(
\lVert \beta _{2^{n+1}}-\beta _{\ast }\rVert \geq 2^{-\rho n}\right) \right]
+E\left[ \lVert \beta _{2^{n+1}}-\beta _{\ast }\rVert ^{2p/(p-2)}I\left(
\lVert \beta _{2^{n+1}}-\beta _{\ast }\rVert <2^{-\rho n}\right) \right]
\notag \\
\leq & c^{\prime }\exp \left( -\alpha _{0}(1-2\rho )n\right) +2^{-\frac{%
2p\rho n}{p-2}}=O\left( 2^{-\frac{2p\rho n}{p-2}}\right) .
\end{align}%
Combining them together we get $E\left[ \bar{R}_{n}\bar{R}_{n}^{T}\right]
=O\left( 2^{-(1+2\rho )n}\right) $. Similarly we can get
\begin{align*}
E\left[ \bar{R}_{n,(\beta ,\beta )}\bar{R}_{n,(\beta ,\beta )}^{T}\right] &
=O\left( 2^{-4\rho n}\right) , \\
E\left[ \bar{R}_{n,(\lambda ,\lambda )}\bar{R}_{n,(\lambda ,\lambda )}^{T}%
\right] & =O\left( 2^{-4\rho n}\right) , \\
E\left[ \bar{R}_{n,(\beta ,\lambda )}\bar{R}_{n,(\beta ,\lambda )}^{T}\right]
& =O\left( 2^{-4\rho n}\right) .
\end{align*}

Because
\begin{align*}
&\Lambda_1\left(\beta_{2^{n+1}}-\frac{1}{2}\left(\beta_{2^n}^O+\beta_{2^n}^E%
\right)\right)+\Lambda_2\left(\lambda_{2^{n+1}}-\frac{1}{2}%
\left(\lambda_{2^n}^O+\lambda_{2^n}^E\right)\right) \\
=&\bar{R}_{n+1}-\frac{1}{2}\left(\bar{R}_{n}^O+\bar{R}_{n}^E\right)+\bar{R}%
_{n+1,(\beta,\beta)}-\frac{1}{2}\left(\bar{R}^O_{n,(\beta,\beta)}+\bar{R}%
^E_{n,(\beta,\beta)}\right) \\
& + \bar{R}_{n+1,(\lambda,\lambda)} -\frac{1}{2}\left(\bar{R}%
^O_{n,(\lambda,\lambda)}+\bar{R}^E_{n,(\lambda,\lambda)}\right)+ \bar{R}%
_{n+1,(\beta,\lambda)}-\frac{1}{2}\left(\bar{R}^O_{n,(\beta,\lambda)}+\bar{R}%
^E_{n,(\beta,\lambda)}\right),
\end{align*}
we have that $E\left[\bar{\Delta}_n\bar{\Delta}_n^T\right]=O\left(2^{-4\rho
n}\right)$. Note that $\rho\in(1/4,1/2)$, it satisfies Assumption \ref%
{asu-general1} of the general principles of unbiased estimators in Section %
\ref{sec:general-principle}.

The computational cost for producing $\bar{\Delta}_n$, denoted by $C_n$, is
of order $O(2^n)$. After generating $2^{n+1}$ iid copies of $X$'s, we can
use Newton's method or other root-finding algorithms to solve the KKT
condition for optimal solution, or use other classic tools such as
subgradient method or interior point method.
\end{proof}

\subsection{Unbiased estimator of optimal value}

\label{unbiased-value}

\begin{theorem}
\label{unbiased-optimal} If Assumption \ref{convex-assumptions} is in force,
then $E\left[Z\right]=f_*$, $Var(Z)<\infty$ and the computation complexity
required to produce $Z$ is bounded in expectation.
\end{theorem}

\begin{proof}
Finite expected computation complexity of producing $\Delta_n$ has been
discussed in the proof to Theorem \ref{unbiased-optimizer}. We now show the
unbiasedness of estimator $Z$. Since $f_{n}(\beta)=\frac{1}{n}%
\sum_{i=1}^nF(\beta,X_i)\to f(\beta)$, uniformly on $\mathcal{D}$, with the
result of Proposition 5.2 in \cite{Shapiro:2009} we have $\hat{f}_n\to f_*$
w.p.1 as $n\to\infty$. If Assumption \ref{asu-cvx-mgf} is in force, $\{%
\hat{f}_{2^n}:n\ge0\}$ is uniformly integrable, hence
\begin{equation*}
E\left[Z\right]=\lim_{n\to\infty}E\left[\hat{f}_{2^n}\right]=f_*.
\end{equation*}

We next prove the estimator $Z$ has finite variance by showing $E\left[%
\Delta^2\right]=O(2^{-4\rho n})$ with $\rho>1/4$. By Taylor expansion around
the unique true optimal solution $\beta_*$ and the KKT condition,
\begin{align*}
\Delta_n=& f_{2^{n+1}}\left(\beta_{2^{n+1}}\right)-\frac{1}{2}%
\left(f_{2^n}^O\left(\beta_{2^n}^O\right)+f_{2^n}^E\left(\beta_{2^n}^E%
\right)\right) \\
=&\frac{1}{2^{n+1}}\sum_{i=1}^{2^{n+1}}F\left(\beta_{2^{n+1}},X_i\right)-%
\frac{1}{2}\left(\frac{1}{2^n}\sum_{i=1}^{2^n}F\left(\beta_{2^n}^O,X_i^O%
\right)+\frac{1}{2^n}\sum_{i=1}^{2^n}F\left(\beta^E_{2^n},X_i^E\right)\right)
\\
=&\left(\frac{1}{2^{n+1}}\sum_{i=1}^{2^{n+1}}\nabla_{\beta}F\left(%
\beta_*,X_i\right)+\nabla
G\left(\beta_*\right)\lambda_*\right)^T\left(\beta_{2^{n+1}}-\beta_*%
\right)+R_{n+1} \\
&-\frac{1}{2}\left(\frac{1}{2^n}\sum_{i=1}^{2^n}\nabla_{\beta}F\left(%
\beta_*,X_i^O\right)+\nabla
G\left(\beta_*\right)\lambda_*\right)^T\left(\beta^O_{2^n}-\beta_*\right)+
R^O_{n} \\
&-\frac{1}{2}\left(\frac{1}{2^n}\sum_{i=1}^{2^n}\nabla_{\beta}F\left(%
\beta_*,X_i^E\right)+\nabla
G\left(\beta_*\right)\lambda_*\right)^T\left(\beta^E_{2^n}-\beta_*%
\right)+R^E_n \\
&-\lambda_*^T\nabla G\left(\beta_*\right)^T\left(\beta_{2^{n+1}}-\frac{1}{2}%
\left(\beta_{2^n}^O+\beta_{2^n}^E\right)\right),
\end{align*}
where $R_{n+1}=O\left(\left\|\beta_{2^{n+1}}-\beta_*\right\|^2\right)$, $%
R^O_{n}=O\left(\left\|\beta_{2^n}^O-\beta_*\right\|^2\right)$ and $%
R^E_n=O\left(\left\|\beta_{2^n}^E-\beta_*\right\|^2\right) $. By using the
moderate LDP explained in the proof of Theorem \ref{unbiased-optimizer},
with $\rho\in(1/4,1/2)$, we have
\begin{equation*}
E\left[R_n^2\right] \le c_1E\left[\left\|\beta_{2^{n}}-\beta_*\right\|^4I%
\left(\lVert\beta_{2^n}-\beta_*\rVert\ge 2^{-\rho n}\right)\right]+c_2E\left[%
\left\|\beta_{2^{n}}-\beta_*\right\|^4I\left(\lVert\beta_{2^n}-\beta_*%
\rVert< 2^{-\rho n}\right)\right]=O\left(2^{-4\rho n}\right).
\end{equation*}
Also similar analysis as (\ref{holder-1}) (\ref{holder-2}) and (\ref%
{holder-2}) in the proof of Theorem \ref{unbiased-optimizer} yields
\begin{equation*}
E\left[\left(\left(\frac{1}{2^{n+1}}\sum_{i=1}^{2^{n+1}}\nabla_{\beta}F%
\left(\beta_*,X_i\right)+\nabla
G\left(\beta_*\right)\lambda_*\right)^T\left(\beta_{2^{n+1}}-\beta_*\right)%
\right)^2\right] = O\left(2^{-(1+2\rho)n}\right)
\end{equation*}
Combining the fact that $E\left[\bar{\Delta}_n\bar{\Delta}_n^T\right]%
=O\left(2^{-4\rho n}\right)$, we finally get $E\left[\Delta_n^2\right]%
=O\left(2^{-4\rho n}\right)$.
\end{proof}

\subsection{Applications and numerical examples}

\subsubsection{Linear Regression}

Linear regression is to solve the following optimization problem
\begin{equation}  \label{linear-regression}
\min_{\beta\in\mathbb{R}^{p+1}}MSE=E_{\mu}\left[F(\beta,(X,y))\right]=E_{\mu}%
\left[\left(y-X^T\beta\right)^2\right],
\end{equation}
where $X\in\mathbb{R}^{p+1}$ is called independent variables, whose first
coordinate is $1$, and $y$ is real valued response called dependent
variable. The pair $(X,y)$ is from distribution $\mu$. The goal is to find
the optimal $\beta_*$ that minimizes the mean-squared-error (MSE).

In many of the real-world problems, we normally have the distribution $\mu $
being the empirical measure of all the data available $\{(X_{i},y_{i}):1\leq
i\leq n_{0}\}$, where $n_{0}$ denote the total number of data points we
have. When $n_{0}$ is enormous, it would be difficult and slow to load all
the data and do computation at once. Like we mentioned in previous sections,
we can take a subsample of the whole dataset to solve the corresponding SAA
linear regression, but it results significant estimation bias. With the
unbiased estimators (\ref{cvx-value-estimator}) and (\ref%
{cvx-solution-estimator}), we can take relatively small subsamples and solve
them on multiple processors in parallel, without any bias.

We have $F(\beta,(X,y))=(y-X^T\beta)^2$ strictly convex and twice
continuously differentiable in $\beta$, so the optimizer is unique. To have
all the required conditions listed in Assumption \ref{convex-assumptions}
satisfied, we can let $G\left(\beta\right)=\left(g_1\left(\beta\right),g_1(%
\beta),\ldots, g_{2(p+1)}(\beta)\right)^T$ with $g_{2i-1}(\beta)=e_i^T%
\beta-M $ and $g_{2i}(\beta)=-e_i^T\beta-M$ for $1\le i\le p+1$, with $M>0$
sufficiently large, so that the unique optimizer $\beta_*$ is in the
interior of $\mathcal{D}=\left\{\beta\in\mathbb{R}^{p+1}:G(\beta)\le
0\right\}$. Then all the conditions follow naturally.

The numerical experiment is to test how the unbiased estimators perform on
some real-world dataset. We use Beijing air pollution data (downloaded from
the website of UCI machine learning repository), which has $43824$ data
points, real-valued PM2.5 concentration and $11$ real-valued independent
variables including time of a day, temperature, pressure, wind direction and
speed, etc. We first use the entire dataset to get the true optimal solution
$\beta_*$ and optimal value $f_*$ as baselines of the experiment. Then we
repeat the SAA approach and our unbiased method for $10000$ times; for the
SAA problem, each time we randomly sample a subset of fixed size, while for
the unbiased method we randomly sample a subset of size $2^{N+1}$ with $N$
geometrically distributed in $\{B,B+1,B+2,\ldots\}$. We call such integer
value $B$ ``burning size". In Chapter \ref{sec:general-principle} we have $%
B=0$, which leads to the smallest possible dataset we can get is of size $1$%
. To better control the variance, our experiment uses $B=10$.

\begin{figure}[]
\caption{Linear regression test on Beijing's PM2.5 data}
\label{fig-LR}\center
\includegraphics[width = 1\linewidth]{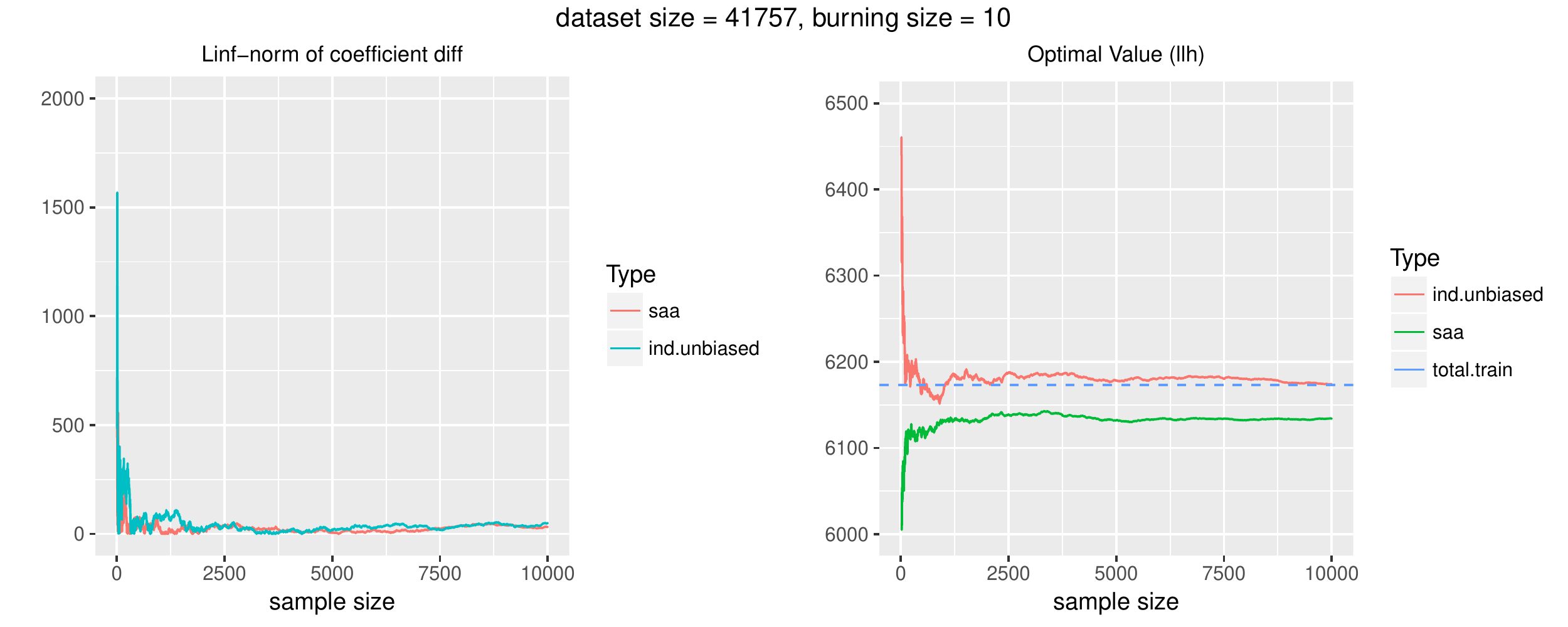}
\end{figure}

The left plot of Figure \ref{fig-LR} has two curves. The red curve shows how
$\left\|\beta_{SAA}-\beta_*\right\|_{\infty}$ changes as we increase the
number of replications, whereas the blue curve shows the same $l_{\infty}$
distance between the mean of the unbiased estimators and $\beta_*$. At the
beginning, both estimators are volatile, though the SAA estimator has
relatively smaller variance than the unbiased estimator, but they stabilize
when the number of replications is around $2500$ and finally are both close
enough to the true optimizer $\beta_*$. The right plot of Figure \ref{fig-LR}
shows how the optimal value estimators from SAA and unbiased method perform
as we increase the number of replications. The blue dashed horizontal line
indicates the level of true optimal value $f_*$, i.e. the MSE computed by
using the entire dataset), the green curve corresponds to the average MSE of
SAA problems and the red curve corresponds to the average MSE of unbiased
method. Clearly the unbiased estimator outperform the other as it gets close
to $f_*$ after some initial fluctuation, however the SAA estimator gives
consistent negative bias, which verifies the theoretic results given in the
SAA literature as we mentioned earlier.

\subsubsection{Regularized Regressions}

Two classic regularized regression techniques are Ridge and LASSO, with $l_2$
and $l_1$ penalty added to the target value function respectively. Ridge
regression is to solve the following optimization problem
\begin{equation}  \label{ridge}
\min_{\beta\in\mathbb{R}^{p+1}}E_{\mu}\left[\left(y-X^T\beta\right)^2\right]%
+\lambda\left\|\beta\right\|_2^2,
\end{equation}
where $\lambda\ge0$ is the shrinkage parameter. An equivalent way to express
the Ridge regression is
\begin{equation*}
\begin{matrix}
\displaystyle \min & E_{\mu}\left[\left(y-X^T\beta\right)^2\right] \\
\text{s.t.} & \beta^T\beta\le t,%
\end{matrix}%
,
\end{equation*}
where $t\propto 1/\lambda$ has one-to-one correspondence with each shrinkage
parameter $\lambda$ in (\ref{ridge}).

Similarly we can express LASSO as
\begin{equation}  \label{lasso}
\begin{matrix}
\displaystyle \min & E_{\mu}\left[\left(y-X^T\beta\right)^2\right] &  \\
\text{s.t.} & (-1)^{r_1}\beta_1+\ldots+(-1)^{r_{p}}\beta_p\le t, &
r_i\in\{0,1\}\ \mbox{ for all }\ i=1,\ldots,p,%
\end{matrix}%
\end{equation}
with $t\ge0$.

For both problem we can verify the conditions in Assumptions \ref%
{convex-assumptions} and use the method proposed to construct unbiased
estimators.

\subsubsection{Logistic Regression}

Logistic regression is to solve the following optimization problem
\begin{equation}  \label{logistic-form}
\min_{\beta\in\mathbb{R}^n}f(\beta)=E_{\mu}\left[F\left(\beta,(X,y)\right)%
\right]=E_{\mu}\left[-\log\left(1+\exp\left(-y\beta^TX\right)\right)\right],
\end{equation}
where $X\in\mathbb{R}^{p+1}$ has its first coordinate being $1$, and $%
y\in\{-1,1\}$ is the label of the class that the data point $(X,y)$ falls
in. The pair $(X,y)$ is from some distribution $\mu.$ The classic logistic
regression is to find the optimal coefficient $\beta_*$ to maximize the
log-likelihood, and we give in (\ref{logistic-form}) an equivalent problem
to minimize the negative of the log-likelihood, i.e., $\min f(\beta)$ with $%
f $ being strict convex.

We run an numerical experiment to check how our unbiased estimators perform,
compared to the SAA estimators of both the optimal solution and the optimal
objective value. The dataset we use is some online advertising campaign data
from Yahoo research, which has $2801523$ data points, each has $22$
real-valued features and one response $y\in\{-1,1\}$ indicating whether it
is a click or not. We first use the entire dataset to get the true optimal
solution $\beta_*$ and optimal value $f_*$ as baselines. Then, for the SAA
method and our unbiased estimating method, we run $10000$ replications each
to see whether they are able to produce a good estimation to $\beta_*$ and $%
f_*$. Again we use the burning size $B$ equal to $10$ here.

\begin{figure}[]
\caption{Logistic regression test on AOL's campaign data}
\label{fig-logistic}\center
\includegraphics[width =
1\linewidth]{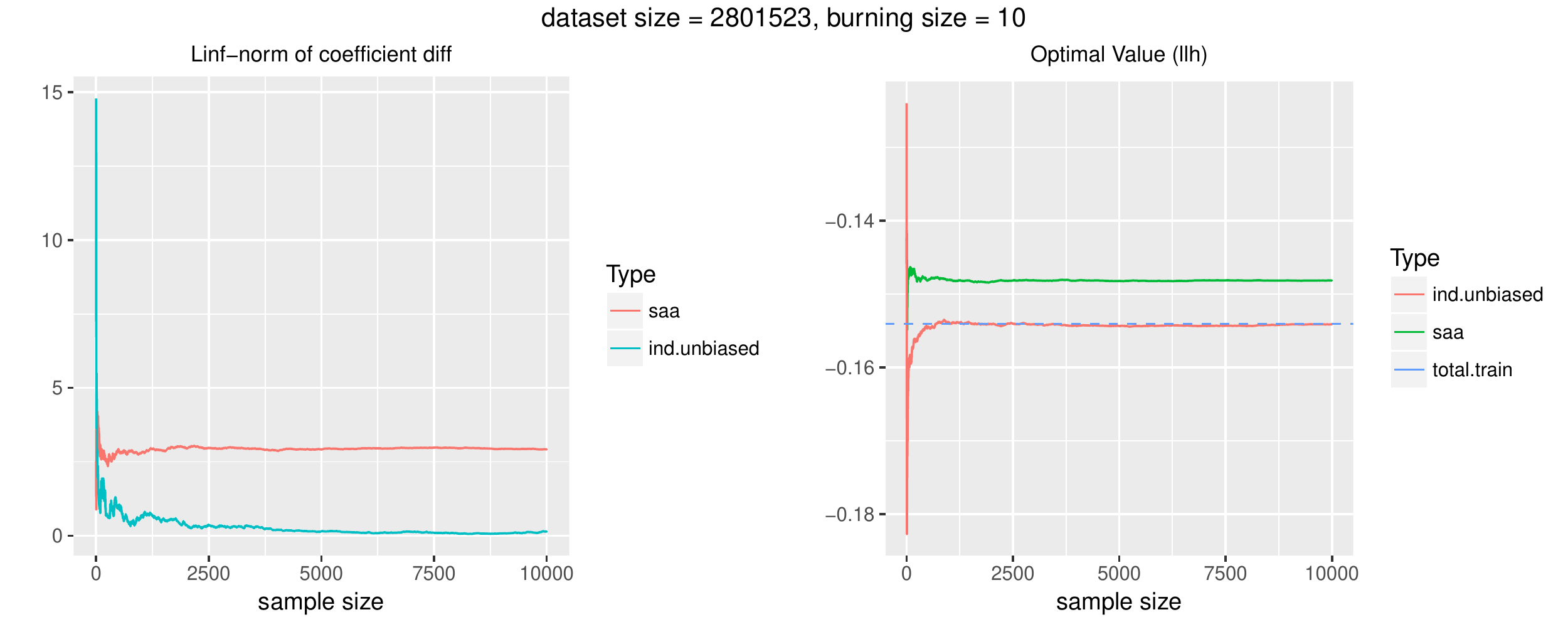}
\end{figure}

In Figure \ref{fig-logistic}, the left plot shows how SAA estimator (in red)
and the unbiased estimator (in blue) approach the true optimal solution $%
\beta_*$ as we increase the size of replications, and similarly the right
plot shows the performance of both estimators for the optimal value $f_*$,
which is represented by the level of the blue dashed line. In both cases,
our unbiased estimators beat the SAA estimators in terms of unbiasedness.

\section{Quantile Estimation}

\label{sec:quantile}

Suppose $\left( X_{k}:k\geq 1\right) $ are iid with cumulative distribution
function $F\left(x\right)=\mu ((-\infty ,x])=P\left( X\leq x\right) $ for $%
x\in \mathbb{R}$. We define $x_{p}=x_{p}\left( \mu \right) =\inf \{x\geq
0:F\left( x\right) \geq p\}$ to be the $p$-quantile of distribution $\mu$
for any given $0<p<1$. If $F\left( \cdot \right) $ is continuous we have
that
\begin{equation*}
F(x_{p})=p.
\end{equation*}%
Connecting to the general framework from Section \ref{sec:general-principle}%
, here we have $\theta \left( \mu \right) :=x_{p}\left( \mu \right) $.

We first impose some assumptions.

\begin{asu}
\label{quantile-assumptions} Distributional quantile assumptions:

\refstepcounter{subassumption} \thesubassumption~\ignorespaces\label{asu-q1}
$F$ is at least twice differentiable in some neighborhood of $x_p$,

\refstepcounter{subassumption} \thesubassumption~\ignorespaces\label{asu-q2}
$F^{\prime \prime }(x)$ is bounded in the neighborhood,

\refstepcounter{subassumption} \thesubassumption~\ignorespaces\label{asu-q3}
$F^{\prime }(x_p)=f(x_p)>0$,

\refstepcounter{subassumption} \thesubassumption~\ignorespaces\label{asu-q4}
$E\left[X^2\right]<\infty$.
\end{asu}

Note that Assumptions \ref{asu-q1}, \ref{asu-q2} and \ref{asu-q3} ensure $%
x_p $ is the unique $p$-quantile of distribution $\mu$. By Bahadur
representation of sample quantiles in \cite{Bahadur_1966}, we have
\begin{equation}  \label{Bahadur-rep}
Y_n=x_p+\frac{np-Z_n}{nf_{\mu}\left(x_p\right)}+R_n,
\end{equation}
where
\begin{equation}  \label{e-sample-quantile}
Y_n=(1-w_n)X_{[np]}+w_nX_{[np]+1},~~~w_n=np-[np]\in[0,1),
\end{equation}
i.e., the sample $p-$quantile of sample $\left(X_1,\ldots,X_n\right)$, $%
Z_n=\sum_{i=1}^nI\left(X_i\le x_p\right)$ and $R_n=O\left(n^{-3/4}\log
n\right)$ as $n\to\infty$ almost surely.

\begin{lemma}
\label{quantile-ui} If Assumption \ref{asu-q4} is in force, $\sup_{n\ge1/p}E%
\left[Y_n^2\right]<\infty$.
\end{lemma}

\begin{proof}
Just follow Bahadur's proof. Let
\begin{equation*}
G_n(x,\omega)=\left(F_n(x,\omega)-F_n(x_p,\omega)\right)-\left(F(x)-F(x_p)%
\right),
\end{equation*}
and let $I_n$ be an open interval $(x_p-a_n,x_p+a_n)$ with the constant $%
a_n\sim \log n/\sqrt{n}$ as $n\to\infty$. Define
\begin{equation*}
H_n(\omega)=\sup\left\{ |G_n(x,\omega)|:x\in I_n\right\}.
\end{equation*}

By Lemma 1 in \cite{Bahadur_1966}, $H_n(\omega)\le K_n(\omega)+\beta_n$ with
$\beta_n=O\left(n^{-3/4}\log n\right)$, $\sum_n\mathbb{P}\left(K_n\ge%
\gamma_n\right)<\infty$ and $\gamma_n=cn^{-3/4}\log n$. By Lemma 2 in \cite%
{Bahadur_1966} we have $Y_n\in I_n$ for sufficiently large $n$ w.p.1. Let $%
n_*=\sup_n\{K_n\ge\gamma_n~~\mbox{or}~~Y_n\notin I_n\}<\infty$, then for all
$n\ge1/p$
\begin{equation*}
E\left[Y_n^2\right]=E\left[Y_n^2I\left(n\le n_*\right)\right]+E\left[%
Y_n^2I\left(n>n_*\right)\right],
\end{equation*}
where
\begin{equation*}
E\left[Y_n^2I\left(n\le n_*\right)\right]\le E\left[\sum_{i=1}^nX_i^2I%
\left(n\le n_*\right)\right]\le n_*E\left[X^2\right]<\infty,
\end{equation*}
and
\begin{align*}
E\left[Y_n^2I\left(n>n_*\right)\right]&=E\left[\left(x_p+\frac{np-Z_n}{%
nf\left(x_p\right)}+R_n\right)^2I\left(n>n_*\right)\right] \\
&\le3x_p^2+\frac{3pq}{nf\left(x_p\right)^2}+3E\left[R_n^2I\left(n>n_*\right)%
\right] \\
&\le3x_p^2+\frac{3pq}{nf\left(x_p\right)^2}+3E\left[H_n^2I\left(n>n_*\right)%
\right] \\
&\le3x_p^2+\frac{3pq}{nf\left(x_p\right)^2}+6\gamma_n^2+6\beta_n^2<\infty.
\end{align*}

Combining these two parts together we can conclude $\sup_n E\left[Y_n^2%
\right]<\infty$.
\end{proof}

We let $Y_{2^{n+1}}$ denote the sample $p-$quantile of $\left(X_1,%
\cdots,X_{2^{n+1}}\right)$, let $Y^O_{2^n}$ denote the sample $p-$quantile
of the odd indexed sub-sample $\left(X^O_1,\cdots,X^O_{2^n}\right)$ and let $%
Y^E_{2^n}$ denote the sample $p-$quantile of the even indexed sub-sample $%
\left(X^E_1,\cdots,X^E_{2^n}\right)$. Then, define
\begin{equation}  \label{quantile-delta}
\Delta_n=Y_{2^{n+1}}-\frac{1}{2}\left(Y^O_{2^n}+Y^E_{2^n}\right).
\end{equation}
Let $n_b=\min\{n\in\mathbb{N}:n\ge1/p\}$. We let the geometrically
distributed random variable $N$ to take values on $\{n_b,n_b+1,\ldots\}$
with $p\left(n\right)=P \left(N=n\right)>0$ for all $n\ge n_b$. Define the
estimator to be
\begin{equation}  \label{quantile-estimator}
Z=\frac{\Delta_N}{p\left(N\right)}+Y_{2^{n_b}}.
\end{equation}

\begin{theorem}
\label{thm:quantile} If Assumption \ref{quantile-assumptions} are in force,
then $E\left[Z\right]=x_p$, $Var(Z)<\infty$ and the computation complexity
required to produce $Z$ is bounded in expectation.
\end{theorem}

\begin{proof}
We first show the unbiasedness of $Z$. Uniform integrability of $%
\{Y_{2^n}:n\ge n_b\}$ is established in Lemma \ref{quantile-ui} with
Assumption \ref{asu-q4} holds true, so we have
\begin{equation*}
E\left[Z\right]=\sum_{n=n_0}^{\infty}E\left[\Delta_n\right]+E\left[%
Y_{2^{n_0}}\right] =\lim_{n\to\infty} E\left[Y_{2^n}\right]= E\left[%
\lim_{n\to\infty}Y_{2^n}\right]=x_p.
\end{equation*}

We next show $Var(Z)<\infty$. With (\ref{Bahadur-rep}) we have
\begin{align*}
\Delta_n&=\left(x_p+\frac{2^{n+1}p-Z_{2^{n+1}}}{2^{n+1}f\left(x_p\right)}%
+R_{2^{n+1}}\right)-\frac{1}{2}\left[\left(x_p+\frac{2^np-Z^O_{2^n}}{2^n
f\left(x_p\right)}+R^O_{2^n}\right)+\left(x_p+\frac{2^np-Z^E_{2^n}}{2^n
f\left(x_p\right)}+R^E_{2^n}\right)\right] \\
&=R_{2^{n+1}}-\frac{1}{2}\left(R^O_{2^n}+R^E_{2^n}\right) \\
&=O\left(n\cdot2^{-3n/4}\right)~~~~w.p.1,
\end{align*}
thus
\begin{equation}  \label{quantile-delta-square}
\Delta_n^2=O\left(n^2\cdot2^{-3n/2}\right).
\end{equation}
If we choose $p(n)=r(1-r)^{n-n_b}$ with $r<1-\frac{1}{2\sqrt{2}}$ for $n\ge
n_b$, then
\begin{equation*}
E\left[\left|\frac{\Delta_N}{p(N)}\right|^2\right]=\sum_{n=n_b}^{\infty}%
\frac{E[\Delta_n^2]}{p(n)}<\infty,
\end{equation*}
hence $Var(Z)<\infty$.

Finally we show the computation cost of generating $\Delta_n$ is finite in
expectation. Each replication of $Z$ involves simulating $2^{N+1}$
independent copies of $X$. If we adopt the selection method based on random
partition introduced in \cite{BFPRT_1973}, then it will cost us $%
O\left(2^{N+1}\right)$ time to identify the sample $p-$quantiles $%
Y_{2^{N+1}} $, $Y^O_{2^N}$, and $Y^E_{2^N}$. Therefore by letting $N$ be an
independent geometrically distributed random variable with success parameter
$r\in\left(1/2,1-2^{-3/2}\right)$, $Z$ is an unbiased estimator of the true
unique $p$-quantile $x_p$ and it has finite work-normalized variance.
\end{proof}

\section*{Acknowledgement} Support from the NSF via grants DMS-1720451, 1838576, 1820942 and CMMI-1538217 is gratefully acknowledged.

\bibliographystyle{plain}
\bibliography{references}

\end{document}